\documentclass[leqno,final]{siamltex}
\usepackage{amsmath}
\usepackage{graphicx,color}
\usepackage{amssymb}

\newtheorem{remark}{Remark}

\newcommand{\Ga}{\Gamma}
\newcommand{\pa}{\partial}

\newcommand{\hP}{\widehat{P}}
\newcommand{\FE}{\mbox{\tiny FE}}

\newcommand{\cE}{\mathcal{E}}
\newcommand{\cT}{\mathcal{T}}

\newcommand{\eps}{\epsilon}

\newcommand{\Ome}{\Omega}
\newcommand{\p}{\partial}
\newcommand{\nab}{\nabla}

\def\esssupI{\underset{t\in [0,\infty)}{\mbox{\rm ess sup }}}

\begin{document}

\title{Analysis of mixed interior penalty discontinuous Galerkin methods
for the Cahn-Hilliard equation and the Hele-Shaw flow}
\markboth{XIAOBING FENG AND YUKUN LI AND YULONG XING}{DG METHODS FOR CAHN-HILLIARD EQUATION}

\author{
Xiaobing Feng\thanks{Department of Mathematics, The University of Tennessee,
Knoxville, TN 37996, U.S.A. ({\tt xfeng@math.utk.edu.}) The work of this author was partially supported 
by the NSF grants DMS-1016173 and DMS-1318486.}
\and
Yukun Li\thanks{Department of Mathematics, The University of Tennessee,
Knoxville, TN 37996, U.S.A. ({\tt yli@math.utk.edu.}) The work of this author was partially supported 
by the NSF grants DMS-1016173 and DMS-1318486.}
\and
Yulong Xing\thanks{Department of Mathematics, The University of Tennessee, 
Knoxville, TN 37996 and Computer Science and Mathematics Division, Oak Ridge National Laboratory, Oak
Ridge, TN 37831 ({\tt xingy@math.utk.edu.}) The work of this author was partially supported by 
the NSF grant DMS-1216454, and Oak Ridge National Laboratory, 
managed by UT-Battelle, LLC for the U. S. Department of Energy under Contract No. De-AC05-00OR22725.}
}

\maketitle

\begin{abstract}
This paper proposes and analyzes two fully discrete mixed interior penalty discontinuous Galerkin 
(DG) methods for the fourth order nonlinear Cahn-Hilliard equation. Both methods use the backward Euler 
method for time discretization and interior penalty discontinuous Galerkin methods for spatial 
discretization.  They differ from each other on how the nonlinear term is treated, one of them is 
based on fully implicit time-stepping and the other uses the energy-splitting time-stepping.
The primary goal of the paper is to prove the convergence of the numerical interfaces of 
the DG methods to the interface of the Hele-Shaw flow. This is achieved 
by establishing error estimates that depend on $\epsilon^{-1}$ only in some low polynomial orders,
instead of exponential orders.
Similar to \cite{Feng_Prohl04}, the crux is to prove a discrete spectrum estimate in the discontinuous 
Galerkin finite element space. However, the validity of such a result is not obvious 
because the DG space is not a subspace 
of the (energy) space $H^1(\Ome)$ and it is larger than the finite element space. This 
difficult is overcome by a delicate perturbation argument which relies on the discrete 
spectrum estimate in the finite element space proved in \cite{Feng_Prohl04}.
Numerical experiment results are also presented to gauge the theoretical results and
the performance of the proposed fully discrete mixed DG methods.
\end{abstract}



\section{Introduction}\label{sec-1}
This paper concerns with mixed interior penalty discontinuous Galerkin (MIP-DG) approximations 
of the following Cahn-Hilliard problem:
\begin{alignat}{2}
u_t-\Delta w &=0  &&\quad \mbox{in } \Omega_T:=\Omega\times(0,T),\label{eq1.1}\\
-\epsilon\Delta u +\frac{1}{\epsilon}f(u) &=w &&\quad \mbox{in } \Omega_T,\label{eq1.1a}\\
\frac{\partial u}{\partial n}
=\frac{\partial w}{\partial n} &=0 &&\quad \mbox{on } \p\Omega_T:=\p\Omega\times(0,T),
\label{eq1.2}\\
u &=u_0 &&\quad \mbox{in } \Omega\times\{t=0\}.\label{eq1.3}
\end{alignat}
Here $\Omega\subseteq \mathbf{R}^d\ (d=2,3)$ is a bounded domain, and $f(u)=F'(u)$, $F(u)$ 
is a nonconvex potential density function which takes its global minimum zero at $u=\pm 1$. 
In this paper, we only consider the following quartic potential density function:
\begin{equation}\label{eq1.4}
F(u)=\frac{1}{4}(u^2-1)^2.
\end{equation}

After eliminating the intermediate variable $w$ (called the chemical potential),
the above system reduces into a fourth order nonlinear PDE for $u$, which is known
as the Cahn-Hilliard equation in the literature.
This equation was originally introduced by John W. Cahn and John E. Hilliard in 
\cite{Cahn_Hilliard58} to describe the process of phase separation, by which the 
two components of a binary fluid spontaneously separate and form domains pure in each 
component. Here $u$ and $1-u$ denote respectively the concentrations of the two fluids, 
with $u=\pm 1$ indicating domains of the two components. We note that the equation 
\eqref{eq1.1}--\eqref{eq1.1a} differs from the original Cahn-Hilliard equation in the 
scaling of the time, and $t$ here corresponds to $\frac{t}{\epsilon}$ in the original 
formulation. $\epsilon$, which is positively small, is called the interaction length. 

Besides its important role in materials phase transition, the Cahn-Hilliard equation 
has been extensively studied due to its close relation with the Hele-Shaw problem. 
It was first formally proved by Pego \cite{Pego89} that the chemical potential 
$w:=-\epsilon\Delta u+\frac{1}{\epsilon}f(u)$ tends to a limit which satisfies 
the following free boundary problem known as the Hele-Shaw problem:
\begin{alignat}{2}
\Delta w &=0
&&\quad \mbox{in } \Omega\setminus\Gamma_t,\ t\in[0,T], \label{eq1.5}\\
\frac{\partial w}{\partial n} &=0
&&\quad \mbox{on } \p\Omega,\ t\in[0,T], \label{eq1.6}\\
w &=\sigma\kappa
&&\quad \mbox{on } \Gamma_t,\ t\in[0,T],\label{eq1.7}\\
V &=\frac{1}{2} \Bigl[ \frac{\partial w}{\partial n} \Bigr]_{\Gamma_t}
&&\quad \mbox{on } \Gamma_t,\ t\in[0,T],\label{eq1.8}  
\end{alignat}
as $\epsilon\searrow0$, provided that the Hele-Shaw problem has a unique classical solution. Here
\begin{equation}\label{eq1.10}
\sigma=\int_{-1}^{1}\sqrt{\frac{F(s)}{2}}ds.
\end{equation}
$\kappa$ and V represent the mean curvature and the normal velocity of the
interface $\Gamma_t$.  A rigorous justification that $u\rightarrow\pm1$ in the interior or exterior of $\Gamma_t$ 
for all $t\in[0,T]$ as $\epsilon\searrow0$ was given by Stoth \cite{Stoth96} for the radially 
symmetric case, and by Alikakos, Bates and Chen \cite{Alikakos94} for the general case. In 
addition, Chen \cite{Chen96} established the convergence of the weak solution of the
Cahn-Hilliard problem to a weak (or varifold) solution of the Hele-Shaw problem.

Moreover, the Cahn-Hilliard equation (together with the Allen-Cahn equation) has become a fundamental 
equation as well as a building block in the phase field methodology (or the diffuse interface 
methodology) for moving interface and free boundary problems arising from various
applications such as fluid dynamics, materials science, image processing and biology 
(cf. \cite{McFadden02,Feng06} and the references therein). The diffuse interface approach
provides a convenient mathematical formalism for numerically approximating the moving interface problems 
because explicitly tracking the interface is not needed in the diffuse interface formulation. 
The main advantage of the diffuse interface method is its ability to handle with ease singularities of 
the interfaces. Like many singular perturbation problems, the main computational issue is to resolve the (small)
scale introduced by the parameter $\epsilon$ in the equation. Computationally, the problem could become 
intractable, especially in three-dimensional cases if uniform meshes are used. This difficulty is often 
overcome by exploiting the predictable (at least for small $\epsilon$) PDE solution profile and by using 
adaptive mesh techniques (cf. \cite{Feng_Wu08} and the references therein), so fine 
meshes are only used in the diffuse interface region.

Numerical approximations of the Cahn-Hilliard equation have been extensively carried out in the
past thirty years (cf. \cite{Du_Nicolaides91,Elliott_French89,Feng_Prohl04} and 
the references therein).  On the other hand, the majority of these works were done for a 
fixed parameter $\epsilon$. The error bounds, which are obtained using the standard Gronwall 
inequality technique, show an exponential dependence on $1/\epsilon$. Such an estimate is clearly 
not useful for small $\epsilon$, in particular, in addressing the issue whether the
computed numerical interfaces converge to the original sharp interface of the Hele-Shaw problem. 
Better and practical error bounds should only depend on $1/\epsilon$ in some (low) polynomial orders
because they can be used to provide an answer to the above convergence question, which 
in fact is the best result (in terms of $\epsilon$) one can expect. The first such polynomial order 
in $1/\epsilon$ a priori estimate was obtained in \cite{Feng_Prohl05} for mixed finite element 
approximations of the Cahn-Hilliard problem \eqref{eq1.1}--\eqref{eq1.4}. 
In addition, polynomial order in $1/\epsilon$ a posteriori error estimates were obtained in \cite{Feng_Wu08}
for the same mixed finite element methods.  One of the key ideas employed in all these works is to use a 
nonstandard error estimate technique which is based on establishing a discrete spectrum estimate (using its 
continuous counterpart) for the linearized Cahn-Hilliard operator. An immediate corollary of the polynomial 
order in $1/\epsilon$ a priori and a posteriori error estimates is the convergence of the numerical 
interfaces of the underlying mixed finite element approximations to the Hele-Shaw flow before the onset 
of singularities of the Hele-Shaw flow as $\epsilon$ and mesh sizes $h$ and $k$ all tend to zero.

The objectives of this paper are twofold: Firstly, we develop some MIP-DG 
methods and to establish polynomial order in $1/\epsilon$ a priori error 
bounds, as well as to prove convergence of numerical interfaces for the MIP-DG methods. This goal is motivated by the 
advantages of DG methods in regard to designing adaptive mesh methods and algorithms, which is an 
indispensable strategy with the diffuse interface methodology. Secondly, we use the Cahn-Hilliard 
equation as another prototypical model problem \cite{Feng_Li14} to develop new analysis techniques 
for analyzing convergence of numerical interfaces to the underlying sharp interface for DG (and nonconforming 
finite element) discretizations of phase field models. To the best of our 
knowledge, no such convergence result and analysis technique is available in the literature for 
fourth order PDEs. The main obstacle for improving the finite element techniques of \cite{Feng_Prohl05} 
is that the DG (and nonconforming finite element) spaces are not subspaces of $H^1(\Ome)$.  As a result, 
whether the needed discrete spectrum estimate holds becomes a key question to answer.

This paper consists of four additional sections. In section \ref{sec-2} we first collect some a priori error 
estimates for problem \eqref{eq1.1}-\eqref{eq1.4}, which show the explicit dependence
on the parameter $\epsilon$. We then cite two important technical lemmas to be used in the later sections.
One of the lemma states the spectral estimate for the linearized Cahn-Hilliard operator.
In section \ref{sec-3}, we propose two fully discrete MIP-DG schemes for problem \eqref{eq1.1}--\eqref{eq1.4}, 
they differ only in their treatment of the nonlinear term. The first main result of this section 
is to establish a discrete spectrum estimate in the DG space, which mimics 
the spectral estimates for the differential operator and its finite element counterpart.
The second main result of this section is to derive optimal error bounds which depends 
on $1/\epsilon$ only in low polynomial orders for both fully discrete MIP-DG methods. 
In section \ref{sec-4}, using the refined error estimates of section \ref{sec-3}, we prove 
the convergence of the numerical interfaces of the fully discrete MIP-DG methods to 
the interface of the Hele-Shaw flow before the onset of the singularities as $\epsilon,h$ 
and $k$ all tend to zero.  Finally, in section \ref{sec-5} we provide some numerical experiments 
to gauge the performance of the proposed fully discrete MIP-DG methods.

\section{Preliminaries}\label{sec-2}

In this section, we shall collect some known results about problem \eqref{eq1.1}--\eqref{eq1.4} 
from \cite{Chen94,Feng_Prohl04,Feng_Prohl05}, which will be used in sections \ref{sec-3}
and \ref{sec-4}. Some general assumptions on the initial condition, 
as well as some energy estimates based on these assumptions, will be cited. 
Standard function and space notations are adopted in this paper \cite{Adams03,Brenner_Scott08}. 
We use $(\cdot,\cdot)$ and $\|\cdot\|_{L^2}$ to denote the standard inner product and
norm on $L^2(\Omega)$. Throughout this paper, $C$ denotes a generic positive constant 
independent of $\epsilon$, space and time step sizes $h$ and $k$, which may have different
values at different occasions. 

We begin with the following well known fact \cite{Alikakos94} that the Cahn-Hilliard equation 
\eqref{eq1.1}-\eqref{eq1.4} can be interpreted as the $H^{-1}$-gradient flow for the Cahn-Hilliard 
energy functional
\begin{equation}\label{eq2.1}
J_\epsilon(v):= \int_\Omega \Bigl( \frac\eps2 |\nabla v|^2+ \frac{1}{\epsilon} F(v) \Bigr)\, dx
\end{equation}

The following assumptions on the initial datum $u_0$ were made in \cite{Feng_Prohl04}, they were used
to derive a priori estimates for the solution of problem \eqref{eq1.1}--\eqref{eq1.4}.

{\bf General Assumption} (GA)
\begin{itemize}
\item[(1)] Assume that $m_0\in (-1,1)$ where
\begin{equation}\label{eq2.2}
m_0:=\frac{1}{|\Omega|}\int_{\Omega}u_0(x)dx. 
\end{equation}
\item[(2)] There exists a nonnegative constant $\sigma_1$ such that
\begin{equation}\label{eq2.3}
J_{\epsilon}(u_0)\leq C\epsilon^{-2\sigma_1}.
\end{equation}
\item[(3)]
There exists nonnegative constants $\sigma_2$, $\sigma_3$ and $\sigma_4$ such that
\begin{align}\label{eq2.4}
\big\|-\epsilon\Delta u_0 +\epsilon^{-1} f(u_0)\big\|_{H^{\ell}(\Omega)} \leq C\epsilon^{-\sigma_{2+\ell}},\quad
\ell=0,1,2.
\end{align}
\end{itemize}

Under the above assumptions, the following solution estimates were proved in \cite{Feng_Prohl04,Feng_Prohl05}.

\begin{proposition}\label{prop2.1}
The solution $u$ of problem \eqref{eq1.1}--\eqref{eq1.4} satisfies the following energy estimates:
\begin{align}
&\esssupI  \Bigl( \frac{\epsilon}{2}\|\nabla u\|_{L^2}^2 +\frac{1}{\epsilon}\|F(u)\|_{L^1} \Bigr)
+\begin{cases} \int_{0}^{\infty}\|u_t(s)\|_{H^{-1}}^2\, ds\\
\int_{0}^{\infty}\|\nabla w(s)\|_{L^2}^2\,ds 
\end{cases}  
\leq J_{\epsilon}(u_0), \label{eq2.5}\\
&\esssupI\|u\|_{L^4}^4\leq C(1+J_{\epsilon}(u_0)), \label{eq2.6}\\
&\esssupI \|u^2-1\|_{L^2}^2\leq C\epsilon J_{\epsilon}(u_0).  \label{eq2.7}
\end{align}
Moreover, suppose that \eqref{eq2.2}--\eqref{eq2.4} hold, $u_0\in H^4(\Omega)$ and $\p\Omega\in C^{2,1}$, 
then $u$ satisfies the additional estimates:
\begin{align}
&\frac{1}{|\Omega|}\int_{\Omega}u(x,t)\, dx=m_0 \quad\forall t\geq 0, \label{eq2.8}\\
&\int_0^{\infty}\|\Delta u\|_{L^2}^2ds\leq C\epsilon^{-(2\sigma_1+3)}, \label{eq2.9}\\
&\int_0^{\infty}\|\nabla\Delta u\|_{L^2}^2ds\leq C\epsilon^{-(2\sigma_1+5)},\label{eq2.10}\\
&\begin{cases}
\esssupI\|u_t\|_{H^{-1}}^2\\
\esssupI\|\nabla w\|_{L^{2}}^2
\end{cases}
+\epsilon\int_{0}^{\infty}\|\nabla u_t\|_{L^2}^2ds\leq C\epsilon^{-\max\{2\sigma_1+3,2\sigma_3\}},\label{eq2.11}\\
&\esssupI\|\Delta u\|_{L^2}\leq C\epsilon^{-\max\{\sigma_1+\frac{5}{2},\sigma_3+1\}},\label{eq2.12}\\
&\esssupI\|\nabla\Delta u\|_{L^2}\leq C\epsilon^{-\max\{\sigma_1+\frac{5}{2},\sigma_3+1\}},\label{eq2.13}\\
&\begin{cases}
\int_0^{\infty}\|u_t\|_{L^2}^2ds\\
\int_0^{\infty}\|\Delta w\|_{L^2}^2ds
\end{cases}
+\esssupI\epsilon\|\Delta u\|_{L^2}^2\leq C\epsilon^{-\max\{2\sigma_1
+\frac{7}{2},2\sigma_3+\frac{1}{2},2\sigma_2+1\}},\label{eq2.14}\\
&\epsilon\int_0^{\infty}\|\Delta u_t\|_{L^2}^2ds+\esssupI\|u_t\|_{L^2}^2
\leq C\epsilon^{-\max\{2\sigma_1+\frac{13}{2},2\sigma_3+\frac{7}{2},2\sigma_2+4,2\sigma_4\}},\label{eq2.15}\\
&\int_0^{\infty}\|\Delta^{-1} u_{tt}(s)\|_{H^{-1}}^2ds\leq C\epsilon^{-\max\{10\sigma_1+10,4\sigma_1+2\sigma_2+5,2\sigma_3-1\}}.\label{eq2.16}
\end{align}
Furthermore, if there exists $\sigma_5>0$ such that
\begin{equation}\label{eq2.17}
\mathop{\rm{lim}}_{s\rightarrow0^{+}}\limits\|\nabla u_t(s)\|_{L^2}\leq C\epsilon^{-\sigma_5},
\end{equation}
then there hold for $d=2,3$,
\begin{align}
&\esssupI\|\nabla u_t\|_{L^2}^2 + \epsilon\int_0^{\infty}\|\nabla\Delta u_t\|_{L^2}^2 ds 
\leq C\rho_0(\epsilon,d),\label{eq2.18}\\
&\int_0^{\infty}\|u_{tt}\|_{H^{-1}}^2ds \leq C\rho_1(\epsilon,d),\label{eq2.19}\\
&\esssupI\|\Delta^2 u\|_{L^2} \leq C\rho_2(\epsilon),\label{eq2.20}
\end{align}
where
\begin{align*}
\rho_0(\epsilon,d) &:=\epsilon^{-\frac{2}{6-d}\max\{2\sigma_1+5,2\sigma_3+2\}
-\max\{2\sigma_1+\frac{13}{2},2\sigma_3+\frac{7}{2},2\sigma_2+4\}} +\epsilon^{-2\sigma_5}\\
&\qquad 
+\epsilon^{-\max\{2\sigma_1+7,2\sigma_3+4\}}, \\
\rho_1(\epsilon, d) &:= \epsilon \rho_0(\epsilon,d), \\
\rho_2(\epsilon) &:=\epsilon^{-\max\{\sigma_1+5,\sigma_3+\frac{7}{2},\sigma_2+\frac{5}{2},\sigma_4+1\}}.
\end{align*}
\end{proposition}

The next lemma concerns with a lower bound estimate for the principal eigenvalue
of the linearized Cahn-Hilliard operator, a proof of this lemma can be found in \cite{Chen94}. 

\begin{lemma}\label{lem3.4}
Suppose that \eqref{eq2.2}--\eqref{eq2.4} hold. Given a smooth initial curve/surface $\Gamma_0$, 
let $u_0$ be a smooth function satisfying $\Gamma_0 = \{x\in\Omega; u_0(x)=0\}$ and some profile 
described in \cite{Chen94}. Let $u$ be the solution to problem \eqref{eq1.1}--\eqref{eq1.4}. 
Define $\mathcal{L}_{CH}$ as
\begin{equation}\label{eq3.25a}
\mathcal{L}_{CH} := \Delta\left(\eps\Delta-\frac{1}{\eps}f'(u)I\right).
\end{equation}
Then there exists $0<\epsilon_0<<1$ and a positive constant $C_0$ such that the principle 
eigenvalue of the linearized Cahn-Hilliard operator $\mathcal{L}_{CH}$ satisfies 
\begin{equation}\label{eq3.25}
\lambda_{CH}:=\mathop{\inf}_{\substack{0\neq\psi\in H^1{(\Omega)}\\ \Delta w=\psi}}
\limits\frac{\epsilon\|\nabla\psi\|_{L^2}^2+\frac{1}{\epsilon}(f'(u)\psi,\psi)}{\|\nabla w\|_{L^2}^2}\geq -C_0
\end{equation}
for $t\in [0,T]$ and $\eps\in (0,\eps_0)$.\\
\end{lemma}

\begin{remark}
(a) A discrete generalization of \eqref{eq3.25} on $C^0$ finite element spaces was
proved in \cite{Feng_Prohl04,Feng_Prohl05}. It plays a pivotal role in the nonstandard
convergence analysis of \cite{Feng_Prohl04,Feng_Prohl05}. In the next section, we shall
prove another discrete generalization of \eqref{eq2.12} on the DG finite element space.

(b) The restriction on the initial function $u_0$ is needed to guarantee
that the solution $u(t)$ satisfies certain profile at later time $t>0$ which is required
in the proof of \cite{Chen94}. One example of admissible initial functions is
$u_0=\tanh(\frac{d_0(x)}{\eps})$, where $d_0(x)$ stands for the signed distance function
to the initial interface $\Gamma_0$. Such a $u_0$ is smooth when $\Gamma_0$ is smooth.
\end{remark}

\medskip
Next lemma can be regarded as a nonlinear generalization of the classical discrete
Gronwall lemma. It gives an upper bound estimate for a discrete sequence which 
satisfies a nonlinear inequality with Bernoulli-type nonlinearity, which will
be utilized crucially in the next section. A proof of this lemma can be found 
in \cite{Pachpatte} and its differential counterpart can be seen in \cite{Feng_Wu08}.

\begin{lemma}\label{lem2.3}
Let $\{S_{\ell} \}_{\ell\geq 1}$ be a positive nondecreasing sequence, 
$\{b_{\ell}\}_{\ell\geq 1}$ and $\{k_{\ell}\}_{\ell\geq 1}$ be nonnegative sequences, 
and $p>1$ be a constant. If
\begin{eqnarray}\label{eq2.131}
&S_{\ell+1}-S_{\ell}\leq b_{\ell}S_{\ell}+k_{\ell}S^p_{\ell} \qquad\mbox{for } \ell\geq 1,
\\ \label{eq2.141}
&S^{1-p}_{1}+(1-p)\mathop{\sum}\limits_{s=1}^{\ell-1}k_{s}a^{1-p}_{s+1}>0
\qquad\mbox{for } \ell\geq 2,
\end{eqnarray}
then
\begin{equation}\label{eq2.151}
S_{\ell}\leq \frac{1}{a_{\ell}} \Bigg\{S^{1-p}_{1}+(1-p)
\sum_{s=1}^{\ell-1}k_{s}a^{1-p}_{s+1}\Bigg\}^{\frac{1}{1-p}}\qquad\text{for}\ \ell\geq 2,  
\end{equation}
where
\begin{equation}\label{eq2.161}
a_{\ell} := \prod_{s=1}^{\ell-1} \frac{1}{1+b_{s}}. 
\end{equation}
\end{lemma}

\section{Fully discrete MIP-DG approximations}\label{sec-3}
In this section we present and analyze two fully discrete MIP-DG 
methods for the Cahn-Hilliard problem \eqref{eq1.1}--\eqref{eq1.4}. The primary 
goal of this section is to derive error estimates for the DG solutions that depend on $\epsilon^{-1}$ 
only in low polynomial orders, instead of exponential orders. 
As in the finite element case (cf. \cite{Feng_Prohl05}), the crux is to establish a discrete 
spectrum estimate for the linearized Cahn-Hilliard operator on the DG space. 

\subsection{Formulations of the MIP-DG method} \label{sec-3.1}
Let $\mathcal{T}_h=\{K\}_{K\in\Omega}$ be a quasi-uniform triangulation of $\Omega$ parameterized by $h > 0$.
For any triangle/tetrahedron $K \in \mathcal{T}_h$, we define
$h_K$ to be the diameter of $K$, and $h:=\mathop{\max}_{K\in\mathcal{T}_h}h_K$. 
The standard broken Sobolev space is defined as
\begin{equation}\label{eq3.1}
H^s(\mathcal{T}_h):=\bigl\{v\in L^{2}(\Omega);\, \forall K\in\mathcal{T}_h,\, v|_{K}\in H^{s}(K) \bigr\}.
\end{equation}
For any $K \in \mathcal{T}_h$, $P_r(K)$ denotes the set of all polynomials of degree at most $r(\geq 1)$ 
on the element $K$, and the DG finite element space $V_h$ is defined as
\begin{equation}\label{eq3.1a}
V_h:=\bigl\{v\in L^{2}(\Omega);\, \forall K\in\mathcal{T}_h,\, v|_{K}\in P_r(K) \bigr\}.
\end{equation}

Let $L^2_0$ denote the set of functions in $L^2(\Omega)$ with zero mean, and let
$\mathring{V}_h:=V_h\cap L^2_0$. We also define $\mathcal{E}_{h}^{I}$ to be the set of all interior 
edges/faces of $\cT_h$, $\mathcal{E}_{h}^{B}$ to be the set of all boundary edges/faces 
of $\cT_h$ on $\Ga=\pa\Ome$, and $\mathcal{E}_{h}:=\cE_h^I\cup \cE_h^B$.
Let $e$ be an interior edge shared by two elements $K_{1}$ and $K_{2}$. For a scalar 
function $v$, define 
\begin{equation*}
\left\{ v\right\} =\frac{1}{2}( v|_K+v|_{K^\prime}),
\quad [v] =v|_K-v|_{K^\prime},
\quad \text{on } e\in \cE_h^I,
\end{equation*}
where K is $K_{1}$ or $K_{2}$, whichever has the bigger global labeling
and $K^{\prime }$ is the other.
The $L^2$-inner product for piecewise functions over the mesh $\cT_h$ is naturally defined by
\begin{align*}
(v, w)_{\cT_h} :=\sum_{K\in\cT_h} \int_K vwdx.
\end{align*}

Let $0\leq t_0 < t_1 < \cdots < t_M =T$ be a partition of the interval $[0,\,T]$ with time step
$k=t_{n+1} - t_n$. Our fully discrete MIP-DG methods are defined as follows: for any $1\leq m\leq M$, 
$(U^m, W^m)\in V_h\times V_h$ are given by
\begin{alignat}{2}\label{eq3.8}
(d_tU^m,\eta)+a_h(W^m,\eta) &=0  &&\qquad \forall \,\eta\in V_h, \\
\epsilon a_h(U^m,v) + \frac{1}{\epsilon}(f^m,v)-(W^m,v) &=0 &&\qquad \forall\, v\in V_h, \label{eq3.9}
\end{alignat}
where 
\begin{align}\label{eq3.2a}
 a_h(u,v)=&\sum_{K\in\mathcal{T}_{h}}\int_{K}\nabla u\cdot\nabla v\,dx
-\sum_{e\in \mathcal{E}_{h}^I}\int_{e}\{\nabla u\cdot \mathbf{n}_{e}\}[v]\,ds\\
&-\sum_{e\in \mathcal{E}_{h}^I}\int_{e}\{\nabla v\cdot \mathbf{n}_{e}\}[u]\,ds
+\sum_{e\in \mathcal{E}_{h}^I}\int_{e}\frac{\sigma_{e}^{0}}{h_e}[u][v]\,ds,	\nonumber
\end{align}
and $\sigma_e^0>0$ is the penalty parameter.
There are two choices of $f^m$ considered in this paper, namely 
\[
f^m=(U^m)^3-U^{m-1} \qquad\mbox{and}\qquad  f^m=(U^m)^3-U^{m},
\]
which lead to the energy-splitting scheme and fully implicit scheme respectively.
$d_t$ is the (backward) difference operator defined by $ d_tU^m:=(U^m-U^{m-1})/k$ 
and $U^0:=\hP_hu_0$ (or $\widehat{Q}_h u_0$) is the starting value, with 
the finite element $H^1$ (or $L^2$) projection $\hP_h$ (or $\widehat{Q}_h$) 
to be defined below.  We refer to \cite{Feng_Li14} for a discussion why a continuous 
projection is needed for the initial condition.
We remark that only the fully implicit case was considered in \cite{Feng_Prohl04,Feng_Prohl05}
for the mixed finite element method.

In order to analyze the stability of \eqref{eq3.8}--\eqref{eq3.9}, 
we need some preparations. First, we introduce three projection operators that will be needed
to derive the error estimates in section \ref{sec-3.4}.  $P_h: H^{s}(\mathcal{T}_{h})\rightarrow {V}_h$ 
denotes the elliptic projection operator defined by
\begin{equation}\label{eq3.11}
a_{h}(u-P_hu,v_{h})+(u-P_hu,v_{h})=0 \qquad\forall\, v_{h}\in {V}_h,
\end{equation}
which has the following approximation properties (see \cite{Chen_Chen04}):
\begin{align}\label{eq3.12}
\|v-P_h v\|_{L^2(\cT_h)} +h\|\nab(v-P_h v)\|_{L^2(\cT_h)}
&\leq Ch^{\min\{r+1,s\}}\|u\|_{H^s(\cT_h)},\\
\frac{1}{|\ln h|^{\overline{r}}} \|v-P_h v\|_{L^\infty(\cT_h)}
+ h\|\nab(u-P_h u)\|_{L^\infty(\cT_h)}
&\leq Ch^{\min\{r+1,s\}}\|u\|_{W^{s,\infty}(\cT_h)}. \label{eq3.13}
\end{align}
Here $\overline{r}:=\min\{1, r\}-\min\{1, r-1\}$.  

Let $\hP_h: H^{s}(\mathcal{T}_{h})\rightarrow S_h:=V_h\cap C^0(\overline{\Ome})$ denote the 
standard continuous finite element elliptic projection, which is the counterpart of 
projection $P_h$. It has the following well-known property \cite{Feng_Prohl04,Feng_Prohl05}:
\begin{equation}\label{eq3.22bb}
\|u-\hP_hu\|_{L^{\infty}}\leq Ch^{2-\frac{d}{2}}\|u\|_{H^2}.
\end{equation}

Next, for any DG function $\Psi_h \in V_h$, we define its continuous finite element 
projection $\Psi_h^{FE}\in S^h$ by
\begin{equation}\label{eq3.7.1}
\tilde{a}_{h}(\Psi_h^{FE},v_h)=\tilde{a}_{h}(\Psi_h,v_h) \qquad \forall\, v_h\in S_h,
\end{equation}
where
\begin{equation*}\label{eq3.7.2}
\tilde{a}_{h}(u,v)=a_{h}(u,v)+\alpha(u,v),
\end{equation*}
and $\alpha$ is a parameter that will be specified later in section \ref{sec-3.3}.

A mesh-dependent $H^{-1}$ norm will also be needed. To the end, we introduce the inverse discrete 
Laplace operator $\Delta_h^{-1}: V_h\rightarrow \mathring{V}_h$ as follows: given 
$\zeta\in V_h$, let $\Delta_h^{-1}\zeta\in\mathring{V}_h$ such that
\begin{equation}\label{eq3.2}
a_h(-\Delta_h^{-1}\zeta,w_h)=(\zeta,w_h) \qquad \forall\, w_h\in \mathring{V}_h.
\end{equation}
We note that $\Delta_h^{-1}$ is well defined provided that $\sigma_e^0 >\sigma_*^0$ for 
some positive number $\sigma_*^0$ and for all $e\in \mathcal{E}_h$ because this condition 
ensures the coercivity of the DG bilinear form $a_h(\cdot,\cdot)$. 

We then define ``-1'' inner product by
\begin{equation}\label{eq3.3}
(\zeta,\xi)_{-1,h}:= a_h(-\Delta_h^{-1}\zeta,-\Delta_h^{-1}\xi)
=(\zeta,-\Delta_h^{-1}\xi)=(-\Delta_h^{-1}\zeta,\xi),
\end{equation}
and the induced mesh-dependent $H^{-1}$ norm is given by
\begin{equation}\label{eq3.4}
\|\zeta\|_{-1,h}:=\sqrt{(\zeta,\zeta)_{-1,h}}
=\mathop{\sup}_{0\neq\xi\in\mathring{V}_h}\limits\frac{(\zeta,\xi)}{|||\xi|||_a},
\end{equation}
where $|||\xi|||_a:=\sqrt{a_{h}(\xi,\xi)}$.  The following properties can be easily 
verified (cf. \cite{Aristotelous12}):
\begin{alignat}{2}\label{eq3.5}
|(\zeta,\xi)| &\leq\|\zeta\|_{-1,h}|||\xi|||_a &&\qquad\forall\,\xi\in V_h, \ \zeta\in \mathring{V}_h, \\
\|\zeta\|_{-1,h} &\leq C\|\zeta\|_{L^2} &&\qquad\forall\,\zeta\in \mathring{V}_h, \label{eq3.6}
\end{alignat}
and, if $\mathcal{T}_h$ is quasi-uniform, then
\begin{equation}\label{eq3.7}
\|\zeta\|_{L^2}\leq C h^{-1}\|\zeta\|_{-1,h}\qquad\forall\,\zeta\in \mathring{V}_h.
\end{equation}

\subsection{Discrete energy law and well-posedness}\label{sec-3.2}

In this subsection we first establish a discrete energy law, which mimics the 
differential energy law, for both fully discrete MIP-DG methods defined in
\eqref{eq3.8}--\eqref{eq3.9}. 
Based on this discrete energy law, we prove the existence and uniqueness
of solutions to the MIP-DG methods by recasting the schemes as convex
minimization problems at each time step. It turns out that the energy-splitting 
scheme is unconditionally stable but the fully implicit scheme is only conditionally
stable. 

\begin{theorem}\label{lem3.1}
Let $(U^m,W^m)\in V_h\times V_h$ be a solution to scheme \eqref{eq3.8}--\eqref{eq3.9}. 
The following energy law holds for any $h,k>0:$
\begin{align}\label{eq3.15}
&E_h(U^\ell)+k\sum_{m=1}^\ell \|d_tU^m\|_{-1,h}^2 +k^2\sum_{m=1}^\ell
\Bigg\{\frac{\epsilon}{2}|||d_t U^m|||_a^2 +\frac{1}{4\epsilon}\|d_t(U^m)^2\|_{L^2}^2 \\
&\hskip2cm
+\frac{1}{2\epsilon}\|U^md_tU^m\|_{L^2}^2\pm\frac{1}{2\epsilon}\|d_t U^m\|_{L^2}^2
\Bigg\}=E_h(U^0)\nonumber
\end{align}
for all $1\leq \ell\leq M,$ where
\begin{equation}\label{eq3.16}
E_h(U):=\frac{1}{4\epsilon}\|U^2-1\|_{L^2}^2+\frac{\epsilon}{2}|||U|||_a^2.
\end{equation}
\end{theorem}
Note that the sign ``$\pm$" in \eqref{eq3.15} takes ``$+$" when $f^m=(U^m)^3-U^{m-1}$ and ``$-$" 
when $f^m=(U^m)^3-U^{m}$.

The proof of the above theorem follows from taking $\eta=-\Delta^{-1} d_t U^m$ 
in \eqref{eq3.8} and $v=d_t U^m$ in \eqref{eq3.9}, adding the resulting two equations and
combining like terms. We leave the detailed calculations to the interested reader. 

\begin{corollary}
Let $\sigma_*^0>0$ be a sufficiently large constant. 
Suppose that $\sigma_e^0> \sigma_*^0$ for all $e\in \mathcal{E}_h$. Then scheme 
\eqref{eq3.8}--\eqref{eq3.9} is stable for all $h,k>0$ when $f^m=(U^m)^3-U^{m-1}$ and is 
stable for $h>0$ and $k=O(\epsilon^3)$ when $f^m=(U^m)^3-U^{m}$.
\end{corollary} 

\begin{proof}
The first case holds trivially from \eqref{eq3.15}. In the second case, the ``bad term" $\|d_t U^m\|_{L^2}$ 
can be controlled by the ``good terms" $\|U^m\|_{-1,h}^2$ and $|||U^m|||_a^2$ by using the norm interpolation inequality 
\eqref{eq3.544} provided that $k=O(\epsilon^3)$. 
\end{proof}

\begin{theorem}\label{existence}
Suppose that $\sigma_e^0> \sigma_*^0$ for all $e\in \mathcal{E}_h$. Then scheme
\eqref{eq3.8}--\eqref{eq3.9} has a unique solution $(U^m,W^m)$ at each time step for 
for all $h,k>0$ in the case $f^m=(U^m)^3-U^{m-1}$ and for $h>0$ and $k=O(\epsilon^3)$
in the case $f^m=(U^m)^3-U^{m}$.
\end{theorem}

\begin{proof}
Setting $\eta=-\Delta_h^{-1} v$ in \eqref{eq3.8} we get
\[
\bigl(d_t U^m, v\bigr)_{-1,h} + \bigl( W^m, v\bigr) =0.
\]
Adding the above equation to \eqref{eq3.9} yields
\[
\bigl(d_t U^m, v\bigr)_{-1,h} + \epsilon a_h\bigl(U^m, v\bigr) 
+ \frac{1}{\epsilon} \bigl( f^m, v \bigr) =0.
\]
Hence, $U^m$ satisfies 
\begin{equation}\label{reduced_eq}
\bigl(U^m, v\bigr)_{-1,h} + k\epsilon a_h\bigl(U^m, v\bigr) 
+ \frac{k}{\epsilon} \bigl( f^m, v \bigr) =\bigl(U^{m-1}, v\bigr)_{-1,h}.
\end{equation}

In the case $f^m=(U^m)^3-U^{m-1}$ it is easy to check that \eqref{reduced_eq} can be recast 
as a convex minimization problem (cf. \cite{Aristotelous12,Feng_Li14}) whose well-posedness holds 
for all $h,k>0$. Hence, in this case there is a unique solution $U^m$ to \eqref{eq3.8}--\eqref{eq3.9}.
On the other hand, when $f^m=(U^m)^3-U^{m}$, there is an extra term $-k\epsilon^{-1} (U^m, v)$ comes out 
from the nonlinear term in \eqref{reduced_eq}. This extra term contributes a ``bad term"
$-k\epsilon^{-1}\|U^m\|_{L^2}^2$ to the functional of the minimization problem. Again, 
this term can be controlled by the ``good terms" $\|U^m\|_{-1,h}^2$ and $|||U^m|||_a^2$ in the functional
by using the norm interpolation inequality \eqref{eq3.544}, provided that $k=O(\epsilon^3)$. Hence, 
in the case $f^m=(U^m)^3-U^{m}$, there is a unique solution $U^m$ to \eqref{eq3.8}--\eqref{eq3.9} for all $h>0$ and 
$k=O(\epsilon^3)$. The proof is complete.
\end{proof}

\subsection{Discrete spectrum estimate on the DG space}\label{sec-3.3}
In this subsection, we shall establish a discrete spectrum estimate for the linearized Cahn-Hilliard 
operator on the DG space, which plays a vital role in our error estimates.

To the end, we first state a slightly modified version of a discrete spectrum estimate for 
the linearized Cahn-Hilliard operator on the continuous finite element space first  proved in 
\cite{Feng_Prohl04,Feng_Prohl05}. Due to the close similarity, we omit the proof of this modified
version and refer the interested reader to \cite{Feng_Prohl04,Feng_Prohl05}.

\begin{lemma}\label{lem3.5}
Suppose the assumptions of Lemma \ref{lem3.4} hold, and $C_0$ is the same as in \eqref{eq3.25}. 
$C_1$ and $C_2$ are defined by 
\begin{align}\label{eq3.26} 
&C_1 := \max_{ |\xi|\leq 2C_0} |f^{\prime\prime}(\xi)|, \\
& \label{eq3.27}
\|u-\hP_hu\|_{L^{\infty}((0,T);L^{\infty})} \leq C_2h^{2-\frac{d}{2}}\epsilon^{\min\{-\sigma_1-\frac52,-\sigma_3-1\}}. 
\end{align}
Then there exists $0<\epsilon_1<<1$ such that, for any $\eps\in (0,\eps_1)$, there holds
\begin{align}
\lambda_{CH}^{FE} \equiv\inf_{0\neq\psi_h\in L_0^2(\Omega)\cap S_h }
\frac{\epsilon\|\nabla\psi_h\|_{L^2}^2 +\frac{2-\eps^3}{2\epsilon}
\bigl(f'(\hP_hu)\psi_h,\psi_h\bigr)}{\|\nabla \Delta^{-1}\psi_h\|_{L^2}^2}\geq -(C_0+1),
\label{eq3.29}
\end{align}
provided that $h$ satisfies
\begin{equation}\label{eq3.30}
h^{2-\frac{d}{2}}\leq (C_1C_2)^{-1}\epsilon^{\max\{\sigma_1+\frac{11}{2},\sigma_3+4\}}. 
\end{equation}
Here $\Delta^{-1}: L^2_0(\Ome)\to H^1(\Ome)\cap L^2_0(\Ome)$ denotes the inverse Laplace operator.

\end{lemma}

We are now ready to state the discrete spectrum estimate on the DG space.

\begin{proposition}\label{prop2.3}
Suppose the assumptions of Lemma \ref{lem3.4} hold. Let $u$ be the solution of \eqref{eq1.1}--\eqref{eq1.4} 
and $P_h u$ denote its DG elliptic projection. Assume
\begin{align}\label{eq2.23}
\esssupI\|u\|_{W^{1+r,\infty}} \leq C\epsilon^{-\gamma}, 
\end{align}
for a constant $\gamma$, then there exists $0<\epsilon_2<<1$ and an $\epsilon$-independent and $h$-independent constant $c_0 > 0$, 
such that for any $\epsilon\in(0,\eps_2)$, there holds 
\begin{equation}\label{eq2.24}
\lambda_{CH}^{DG}=\inf_{0\neq\Phi_h\in L^2_0{(\Omega)}\bigcap V_h}\frac{\epsilon a_h(\Phi_h,\Phi_h)
+\frac{1-\eps^3}{\eps}(f^\prime(P_hu)\Phi_h,\Phi_h)}{\|\nabla \Delta^{-1}\Phi_h\|_{L^2}^2}\geq -c_0,
\end{equation}
provided that $h$ satisfies the constraints
\begin{align}\label{eq2.25}
h^{2-\frac{d}{2}}&\leq (C_1C_2)^{-1}\epsilon^{\max\{\sigma_1+\frac{11}{2},\sigma_3+4\}},\\
h^{1+r}|\ln\,h|^{\bar{r}}&\leq (C_1C_3)^{-1}\epsilon^{\gamma+3}.\label{20150129}
\end{align}
where $C_1$ and $C_2$ are same as in Lemma \ref{lem3.5}, $\bar{r}$ and $C_3$ are defined by
\begin{align}\label{eq2.26a}
&\bar{r} =\min\{1,r\}-\min\{1,r-1\},\notag\\
&\|u-P_hu\|_{L^{\infty}((0,T);L^{\infty})} \leq C_3h^{1+r}|\ln\,h|^{\bar{r}}\epsilon^{-\gamma}. \notag 
\end{align}

\end{proposition}

\begin{proof}
By Proposition 2 in \cite{Feng_Prohl04}, under the mesh constraint \eqref{eq2.25}, we have
\begin{equation}\label{eq3.41}
\|f^\prime(\hP_hu)- f^\prime(u)\|_{L^{\infty}((0,T);L^{\infty})}\leq \epsilon^3.
\end{equation}
Similarly, under the mesh condition \eqref{20150129}, we can show that for any $\epsilon>0$, there holds
\begin{equation}\label{eq3.39}
\|f^\prime(P_hu)-f^\prime(u)\|_{L^{\infty}((0,T);L^{\infty})}\leq\epsilon^3.
\end{equation}
It follows from \eqref{eq3.41} and \eqref{eq3.39} that 
\begin{equation}\label{eq3.43}
\|f^\prime(P_hu)- f^\prime(\hP_hu)\|_{L^{\infty}((0,T);L^{\infty})}\leq 2\epsilon^3
 \quad \mbox{and} \quad f^\prime(P_hu)\geq f^\prime(\hP_hu)-2\epsilon^3.
\end{equation}
Therefore,
\begin{align}\label{sec3:spec1}
&\hskip-3mm 
\epsilon a_h(\Phi_h,\Phi_h)+\frac{1-\eps^3}{\epsilon}\bigl(f^\prime(P_hu)\Phi_h,\Phi_h\bigr) \\
&\geq\epsilon a_h(\Phi_h,\Phi_h)+\frac{1-\eps^3}{\epsilon}\bigl(f^\prime(\hP_hu)\Phi_h,\Phi_h\bigr)
-2\eps^2(1-\eps^3)\|\Phi_h\|_{L^2}^2\nonumber\\
&= \epsilon\frac{1-\epsilon^3}{1-\frac{\eps^3}{2}}a_{h}(\Phi_h,\Phi_h)
+\frac{1-\epsilon^3}{\epsilon} \bigl(f^\prime(\hP_hu)\Phi_h,\Phi_h \bigr)  \notag \\
&\hskip2mm
-2\eps^2(1-\epsilon^3)\|\Phi_h\|_{L^2}^2
+\frac{\epsilon^4}{2-\eps^3}a_{h}(\Phi_h,\Phi_h) \nonumber\\
&=\epsilon\frac{1-\epsilon^3}{1-\frac{\eps^3}{2}}a_{h}(\Phi_h,\Phi_h)
+\frac{1-\epsilon^3}{\epsilon}\int_{\Omega}f'(\hP_hu)\Bigl((\Phi_h)^2-(\Phi_h^{FE})^2\Bigr)\,dx\nonumber\\
&\hskip2mm
+\frac{1-\epsilon^3}{\epsilon}\int_{\Omega}f'(\hP_hu)(\Phi_h^{FE})^2dx
-2\epsilon^2(1-\epsilon^3)\|\Phi_h\|_{L^2}^2+\frac{\epsilon^4}{2-\eps^3}a_{h}(\Phi_h,\Phi_h).  \notag
\end{align}

Next, we derive a lower bound for each of the first two terms on the right-hand side of \eqref{sec3:spec1}.
Notice that the first term can be rewritten as 
\begin{align}\label{sec3:spec2}
&a_{h}(\Phi_h,\Phi_h)
=a_{h}(\Phi_h-\Phi_h^{FE},\Phi_h-\Phi_h^{FE})
+2a_{h}(\Phi_h,\Phi_h^{FE}) - a_{h}(\Phi_h^{FE},\Phi_h^{FE})  \\
&\hskip 0.62in
=a_{h}(\Phi_h-\Phi_h^{FE},\Phi_h-\Phi_h^{FE})
+\|\nabla\Phi_h^{FE}\|_{L^2}^2 + 2\alpha \|\Phi_h^{FE}-\Phi_h\|_{L^2}^2 \notag \\
&\hskip 0.8in
+2\alpha \bigl( \Phi_h^{FE}-\Phi_h,\Phi_h \bigr). \notag
\end{align}

To bound $\|\Phi_h-\Phi_h^{FE}\|_{L^2}$ from above, we consider the following auxiliary problem:
\begin{align*} 
\widetilde{a}_h(\phi, \chi) = \bigl(\Phi_h-\Phi_h^{FE}, \chi\bigr) \qquad\forall\, \chi\in H^1(\Ome).
\end{align*}
For $\sigma_e^0>\sigma_*^0$ for all $e\in \mathcal{E}_h$, the above problem has a unique
solution $\phi\in H^{1+\theta}(\Omega)$ for $0<\theta\leq1$ such that 
\begin{equation}\label{eq3.47}
\|\phi\|_{H^{1+\theta}(\Omega)}\leq C\|\Phi_h-\Phi_h^{FE}\|_{L^2}\qquad \text{for}\ \theta\in (0, 1].
\end{equation}
By the definition of $\Phi_h^{\FE}$, we immediately get the following Galerkin orthogonality:
\[
\widetilde{a}_h \bigl(\Phi_h-\Phi_h^{FE}, \chi_h\bigr)=0 \qquad\forall \chi_h\in S_h,
\]
It follows from the duality argument (cf. \cite[Theorem 2.14]{Riviere08}) that 
\begin{align}\label{eq3.50}
\|\Phi_{h}-\Phi_h^{FE}\|_{L^2}^2
&\leq Ch^{2\theta} \widetilde{a}_h(\Phi_{h}-\Phi_h^{FE},\Phi_{h}-\Phi_h^{FE}) \\
&\leq Ch^{2\theta} a_h(\Phi_{h}-\Phi_h^{FE},\Phi_{h}-\Phi_h^{FE}) 
+ Ch^{2\theta} \alpha \|\Phi_{h}-\Phi_h^{FE}\|_{L^2}^2 . \notag
\end{align}
For all $h$ satisfying $Ch^{2\theta}\alpha<1$, we get 
\begin{equation}\label{20150214_1}
\|\Phi_{h}-\Phi_h^{FE}\|_{L^2}^2
\leq \frac{Ch^{2\theta}}{1-Ch^{2\theta}\alpha} a_h(\Phi_{h}-\Phi_h^{FE},\Phi_{h}-\Phi_h^{FE}).
\end{equation}
Now the last term on the right-hand side of \eqref{sec3:spec2} can be bounded as follows:
\begin{align}\label{eq3.50add1}
2\alpha\bigl(\Phi_h^{FE}-\Phi_h,\Phi_h\bigr)&\geq -2\alpha\|\Phi_h^{FE}-\Phi_h\|_{L^2}\|\Phi_h\|_{L^2}\\
&\geq -2\alpha\sqrt{\frac{Ch^{2\theta} a_h(\Phi_h-\Phi^{FE}_h,\Phi_h-\Phi_h^{FE})}{1-Ch^{2\theta}\alpha}} \,\|\Phi_h\|_{L^2}\notag\\
&\geq -\frac12a_h(\Phi_h-\Phi^{FE}_h,\Phi_h-\Phi_h^{FE})
 -\frac{2C\alpha^2h^{2\theta}}{1-Ch^{2\theta}\alpha}\|\Phi_h\|_{L^2}^2. \notag
\end{align}

The second term on the right-hand side of \eqref{sec3:spec1} can be bounded by
\begin{align}
&\hskip-3mm
\int_{\Omega}f'(\hP_hu)\bigl( (\Phi_{h})^2-(\Phi^{FE}_{h})^2 \bigr)\,dx 
\geq -C \int_{\Omega} \bigl|(\Phi_{h})^2-(\Phi^{FE}_{h})^2 \bigr|\,dx \label{sec3:spec3}\\
&\hskip2mm
= -C\int_{\Omega}\Bigl| -\bigl(\Phi_{h}-\Phi_h^{FE}\bigr)^{2} 
+ 2\Phi_h \bigl(\Phi_{h}-\Phi_h^{FE} \bigr) \Bigr|\, dx\nonumber\\
&\hskip2mm
\geq -C\|\Phi_h-\Phi_h^{FE}\|_{L^2}^2
-\frac{\eps^3(1-\epsilon^3)}{1-\frac{\eps^3}{2}}\|\Phi_h\|_{L^2}^2
-C\frac{1-\frac{\eps^3}{2}}{\eps^3(1-\eps^3)}\|\Phi_h-\Phi_h^{FE}\|_{L^2}^2.  \nonumber
\end{align}
Here we have used the facts that 
\begin{equation}\label{eq3.44}
\|u\|_{L^{\infty}((0,T);L^{\infty})}\leq C, \qquad 
|f^{\prime}(\hP_hu)|\leq |f^{\prime}(u)|+\epsilon^3\leq C.
\end{equation}

Substituting \eqref{20150214_1} into \eqref{sec3:spec3} yields
\begin{align}\label{eq3.53}
&\frac{1-\epsilon^3}{\epsilon}\int_{\Omega}f'(\hP_hu)\bigl((\Phi_{h})^2-(\Phi^{FE}_{h})^2\bigr)\,dx  \\
&\hskip5mm
\geq -\gamma_{3}\frac{\eps(1-\epsilon^3)}{1-\frac{\eps^3}{2}} a_h(\Phi_{h}-\Phi_h^{FE},\Phi_{h}-\Phi_h^{FE})
-\frac{\eps^2(1-\epsilon^3)}{1-\frac{\eps^3}{2}} \|\Phi_h\|_{L^2}^2, \notag
\end{align}
where 
\[
\gamma_3\geq \frac{Ch^{2\theta}}{1-Ch^{2\theta}\alpha}\cdot 2C\frac{1-\frac{\eps^3}{2}}{\epsilon(1-\eps^3)} 
\Bigl (1+\frac{1-\frac{\eps^3}{2}}{\eps^3(1-\eps^3)} \Bigr),
\]
and $h$ is chosen small enough such that $\gamma_3<1/4$. 

The term $\|\Phi_h\|_{L^2}^2$ can be bounded by
\begin{align} \label{eq3.544}
\|\Phi_h\|_{L^2}^2 &= (\Phi_h,\Phi_h)=a_h(\Delta_h^{-1}\Phi_h,\Phi_h) 
\leq a_h(\Delta_h^{-1}\Phi_h,\Delta_h^{-1}\Phi_h)^{\frac{1}{2}}a_h(\Phi_h,\Phi_h)^{\frac{1}{2}} \\
&\leq \frac{\rho}{2} a_h(\Delta_h^{-1}\Phi_h,\Delta_h^{-1}\Phi_h)+\frac{1}{2\rho}a_h(\Phi_h,\Phi_h)\notag
\end{align}
for any constant $\rho>0$. 

Adding the fifth term on the right-hand side of \eqref{sec3:spec1},
the last term on the right-hand side of \eqref{eq3.50add1} and that of \eqref{eq3.53}, 
we get for all $h$ satisfying $2C\alpha^2h^{2\theta}/(1-Ch^{2\theta}\alpha) \leq\eps$
\begin{align} \label{20150131_1}
&-\biggl( \frac{\eps(1-\epsilon^3)}{1-\frac{\eps^3}{2}}\,\frac{2C\alpha^2h^{2\theta}}{1-Ch^{2\theta}\alpha}
+\frac{3\eps^2(1-\epsilon^3)}{1-\frac{\eps^3}{2}} \biggr) \|\Phi_h\|_{L^2}^2
\geq -\frac{4\eps^2(1-\epsilon^3)}{1-\frac{\eps^3}{2}}\|\Phi_h\|_{L^2}^2 \\
&\hskip 1.4in
\geq  -\frac{\eps^4}{2(2-\eps^3)}a_h(\Phi_h^{},\Phi_h^{})
-Ca_h(\Delta_h^{-1}\Phi_h^{},\Delta_h^{-1}\Phi_h^{}). \notag
\end{align}

Combining \eqref{sec3:spec2}, \eqref{eq3.50add1}, \eqref{eq3.53} and \eqref{20150131_1} with \eqref{sec3:spec1}, we have
\begin{align} \label{sec3:spec4}
&\hskip-2mm
\epsilon a_{h}(\Phi_h,\Phi_h)+\frac{1-\epsilon^3}{\epsilon}\int_{\Omega}f'(P_hu)(\Phi_{h})^2\,dx   \\
& \geq \frac{\epsilon(1-\epsilon^3)}{4-2\eps^3}
a_h(\Phi_{h}-\Phi_h^{FE},\Phi_{h}-\Phi_h^{FE})
+\frac{2\alpha\epsilon(1-\epsilon^3)}{1-\frac{\eps^3}{2}}\|\Phi_h^{FE}-\Phi_h\|_{L^2}^2
 \notag \\
&+\frac{\eps(1-\epsilon^3)}{1-\frac{\eps^3}{2}}\|\nabla \Phi_h^{FE}\|_{L^2}^2 
-Ca_h(\Delta_h^{-1}\Phi_h,\Delta_h^{-1}\Phi_h) \notag\\
& +\frac{1-\epsilon^3}{\epsilon}\int_{\Omega}f'(\hP_hu)(\Phi^{FE}_{h})^2\,dx
+\frac{\epsilon^4}{2(2-\eps^3)}a_h(\Phi_h,\Phi_h). \notag
\end{align}
Applying the spectrum estimate \eqref{eq3.29}, we get
\begin{align*}
&\epsilon\frac{1-\epsilon^3}{1-\frac{\eps^3}{2}}\|\nabla \Phi_h^{FE}\|_{L^2}^2
+\frac{1-\epsilon^3}{\epsilon}\int_{\Omega}f'(\hP_hu)(\Phi^{FE}_{h})^2\,dx \\
&\hskip5mm
=\frac{1-\epsilon^3}{1-\frac{\eps^3}{2}}\left(\epsilon\|\nabla \Phi_h^{FE}\|_{L^2}^2
+\frac{1-\frac{\eps^3}{2}}{\epsilon}\int_{\Omega}f'(\hP_hu)(\Phi^{FE}_{h})^2dx\right)	\notag	\\
&\hskip5mm
\geq -\frac{1-\epsilon^3}{1-\frac{\eps^3}{2}}(C_0+1)\|\nabla \Delta^{-1}\Phi_h^{FE}\|_{L^2}^2, \notag
\end{align*}
which together with \eqref{sec3:spec4} implies that
\begin{align}\label{sec3:spec5} 
&\epsilon a_{h}(\Phi_h,\Phi_h)+\frac{1-\epsilon^3}{\epsilon}\int_{\Omega}f'(P_hu)(\Phi_{h})^2\,dx \\
&\geq -Ca_h(\Delta_h^{-1}\Phi_h,\Delta_h^{-1}\Phi_h)-C\|\nabla \Delta^{-1}\Phi_h^{FE}\|_{L^2}^2 
+\frac{2\alpha\epsilon(1-\epsilon^3)}{1-\frac{\eps^3}{2}}\|\Phi_h^{FE}-\Phi_h\|_{L^2}^2. \notag
\end{align}

By the stability of $\Delta^{-1}$, we have
\begin{align*}
\|\nabla\Delta^{-1}(\Phi_h-\Phi_h^{FE})\|_{L^2}^2\leq \widehat{C}\|\Phi_h-\Phi_h^{FE}\|_{L^2}^2,
\end{align*}
which together with the triangle inequality yields
\begin{align*}\label{20150206_2}
\|\nabla \Delta^{-1}\Phi_h^{FE}\|_{L^2}^2
&\leq2\|\nabla \Delta^{-1}\Phi_h\|_{L^2}^2 +2\widehat{C} \|\Phi_h-\Phi_h^{FE}\|_{L^2}^2.
\end{align*}
 
Similarly, since $\Delta_h^{-1}\Phi_h$ is the elliptic projection of $\Delta^{-1}\Phi_h$, there holds
\begin{equation*}
a_h(\Delta_h^{-1}\Phi_h,\Delta_h^{-1}\Phi_h)\leq C\|\nabla\Delta^{-1}\Phi_h\|_{L^2}^2.
\end{equation*}
Therefore, choosing $\alpha=O(\widehat{C}\eps^{-1})$, \eqref{sec3:spec5} can be further reduced into
\begin{equation*}
\epsilon a_{h}(\Phi_h,\Phi_h)+\frac{1-\epsilon^3}{\epsilon}\int_{\Omega}f'(P_hu)(\Phi_{h})^2\,dx
\geq -c_0\|\nabla \Delta^{-1}\Phi_h\|_{L^2}^2  \label{sec3:spec5.5}
\end{equation*}
for some $c_0>0$. This proves \eqref{eq2.24}, and the proof is complete.
\end{proof}

\subsection{Error analysis}\label{sec-3.4}
In this subsection, we shall derive some optimal error estimates for the proposed MIP-DG schemes 
\eqref{eq3.8}--\eqref{eq3.9}, in which the constants in the error bounds depend on 
$\eps^{-1}$ only in low polynomial orders, instead of exponential orders. The key to
obtaining such refined error bounds is to use the discrete spectrum estimate \eqref{eq2.24}.
In addition, the nonlinear Gronwall inequality presented in Lemma \ref{lem2.3} also plays an 
important role in the proof. To ease the presentation, we set $r=1$ in this subsection and section \ref{sec-4},
and generalization to $r>1$ can be proven similarly.

The main results of this subsection are stated in the following theorem.

\begin{theorem}\label{thm3.1} 
Let $\{(U^m, W^m)\}_{m=0}^M$ be the solution of scheme (\ref{eq3.8})--(\ref{eq3.9}) with
$r=1$.  Suppose that (GA) holds and $\sigma_e^0>\sigma_*^0$ for all $e\in \mathcal{E}_h$, and define
\begin{align}
&\rho_3(\eps):=\epsilon^{-\max\{2\sigma_1+\frac{13}{2},2\sigma_3+\frac{7}{2},2\sigma_2+4,2\sigma_4\}-4},\label{eq3.64b}\\
&r(h,k; \epsilon ,d,\sigma_i):=k^2\rho_1(\eps;d)+h^6\rho_3(\eps).\label{eq3.64f}
\end{align}
Then, under the following mesh and starting value conditions:
\begin{align}
h^{2-\frac{d}{2}}&\leq (C_1C_2)^{-1}\epsilon^{\max\{\sigma_1+\frac{11}{2},\sigma_3+4\}},\\
h^{1+r}|\ln\,h|^{\bar{r}}&\leq (C_1C_3)^{-1}\epsilon^{\gamma+3},\\
k&\leq\epsilon^3 \qquad \text{ when } f^m=(U^m)^3-U^m, \label{eq3.55}\\
h^{2\theta}&\leq C\frac{\epsilon(1-\epsilon^3)}{8-4\eps^3}, \label{eq3.56}\\
k&\leq C\eps^{\frac{4(6+d)}{4-d}+(4d-2)\sigma_1},\\
%
%
(U^0,1) &=(u_0,1),\label{eq3.58} \\
\|u_0-U^0\|_{H^{-1}} &\leq Ch^{3}\|u_0\|_{H^2}, \label{eq3.59}
\end{align}
there hold the error estimates
\begin{align} \label{eq3.60}
&\max_{0\leq m\leq M}\|u(t_m)-U^m\|_{-1,h}
+\Bigl(\sum_{m=1}^{M} k^2\|d_{t}(u(t_m)-U^m)\|_{-1,h}^2\Bigr)^{\frac{1}{2}} \\
&\hskip 1.5in
\leq Cr(h,k; \epsilon ,d,\sigma_i)^{\frac12},\nonumber\\ 
\label{eq3.60a}
&\Bigl(k\sum_{m=1}^{M} \|u(t_m)-U^m\|_{L^2}^2\Bigr)^{\frac{1}{2}}\\
&\hskip 1.5in
\leq C\Bigl( h^2\epsilon^{-\max\{\sigma_1+\frac{5}{2},\sigma_3+1\}}
+\epsilon^{-2}r(h,k; \epsilon ,d,\sigma_i)^{\frac12}\Bigr), \notag\\
\label{eq3.60b}
&\Bigl(k\sum_{m=1}^{M} \|\nabla\bigl(u(t_m)-U^m\bigl)\|_{L^2}^2\Bigr)^{\frac{1}{2}}\\
&\hskip 1.5in
\leq C\Bigl(h\epsilon^{-\max\{\sigma_1+\frac{5}{2},\sigma_3+1\}}
+\epsilon^{-2}r(h,k; \epsilon ,d,\sigma_i)^{\frac12}\Bigr).\notag
\end{align}
Moreover, if the starting value $U^0$ satisfies
\begin{equation}\label{eq3.61}
\|u_0-U^0\|_{L^2}\leq Ch^2\|u_0\|_{H^2},
\end{equation}
then there hold
\begin{align}\label{eq3.61a}
&\max_{0\leq m\leq M}\|u(t_m)-U^m\|_{L^{2}}+\Bigl(k\mathop{\sum}\limits_{m=1}^{M}
k\|d_t(u(t_m)-U^m)\|_{L^2}^2\Bigr)^{\frac{1}{2}} \\
&\quad
+\Bigl(\frac{k}{\eps}\sum_{m=1}^{M} \|w(t_m)-W^m\|_{L^2}^2\Bigr)^{\frac{1}{2}}\notag\\
&\hskip 0.8in
\leq C\Bigl(h^2\rho_3(\eps)^{\frac12}+\epsilon^{-\frac{7}{2}} r(h,k;\epsilon,d,\sigma_i)^{\frac{1}{2}} \Bigr)\nonumber\\
&\label{eq3.62} 
\max_{0\leq m\leq M}\|u(t_m)-U^m\|_{L^{\infty}}\\
&\hskip 0.8in
\leq C\Bigl(h^{2}|\ln h|\epsilon^{-\gamma} 
+h^{-\frac{d}{2}}\epsilon^{-\frac{7}{2}} r(h,k;\epsilon,d,\sigma_i)^{\frac{1}{2}} \Bigr).\nonumber
\end{align}
Furthermore, suppose that the starting value $W^0$ satisfies
\begin{equation}\label{eq3.63}
\|P_hw_0-W^0\|_{L^2}\leq Ch^{\beta}
\end{equation}
for some $\beta>1$, and there exists a constant $\gamma^{\prime}$ such that
\begin{align}\label{20150208_1}
\esssupI\|w\|_{W^{2,\infty}} \leq C\epsilon^{-\gamma^{\prime}}, 
\end{align}
then we have
\begin{align} \label{eq3.63a}
&\max_{0\leq m\leq M}\|w(t_m)-W^m\|_{L^{2}}
\leq C\Bigl( h^2\rho_3(\eps)+h^{\beta}\\
&\hskip 1.8in
+k^{-\frac{1}{2}}\epsilon^{-3} r(h,k;\epsilon,d,\sigma_i)^{\frac{1}{2}}\Bigr),\nonumber\\
\label{eq3.63b}
&\max_{0\leq m\leq M}\|w(t_m)-W^m\|_{L^{\infty}}
\leq C\Bigl( h^{-\frac{d}{2}} \Bigl( k^{-\frac{1}{2}}\epsilon^{-3} r(h,k;\epsilon,d,\sigma_i)^{\frac{1}{2}}
+h^{\beta} \Bigr) \\
&\hskip 1.8in
 +h^2|\ln h|\eps^{-\gamma^{\prime}}\Bigr). \nonumber
\end{align}
\end{theorem}

\begin{proof}
In the following, we only give a proof for the convex splitting scheme corresponding to
$f^{m} = (u^m)^3 - u^{m-1}$ in \eqref{eq3.4} because the proof for the fully implicit scheme
with $f^m = (u^m)^3 - u^m$ is almost same. Since the proof is long, we divide it into four steps.

\smallskip
{\em Step 1:} 
It is obvious that equations \eqref{eq1.1}--\eqref{eq1.3} imply that
\begin{alignat}{2}\label{eq3.71}
\bigl(u_t(t_m),\eta_h\bigr)+ a_h(w(t_m),\eta_h) &=0 &&\qquad \forall \eta_h\in V_h, \\
\label{eq3.72}
\epsilon a_h(u(t_m),v_h)+\frac{1}{\epsilon}\bigl(f(u(t_m)),v_h\bigr) &=\bigl(w(t_m),v_h\bigr) &&\qquad \forall v_h\in V_h.
\end{alignat}
Define error functions $E^m :=u(t_m)-U^m$ and $G^m :=w(t_m)-W^m$. Subtracting \eqref{eq3.8} from 
\eqref{eq3.71} and \eqref{eq3.9} from \eqref{eq3.72} yield the following error equations:
\begin{alignat}{2}\label{eq3.73}
\bigl(d_t E^m,\eta_h\bigr)+ a_h(G^m,\eta_h) &=\bigl(R(u_{tt},m),\eta_h\bigr) &&\qquad \forall \eta_h\in V_h, \\
\label{eq3.74}
\epsilon a_h(E^m,v_h)+\frac{1}{\epsilon}\bigl(f(u(t_m))-f(U^m),v_h \bigr)&=\bigl(G^m,v_h\bigr) &&\qquad \forall v_h\in V_h,
\end{alignat}
where
\begin{equation*} 
R(u_{tt};m):=\frac{1}{k}\int^{t_m}_{t_{m-1}}(s-t_{m-1})u_{tt}(s)\,ds.
\end{equation*}
It follows from \eqref{eq2.19} that
\begin{align*}
k\sum_{m=1}^M\|R(u_{tt};m)\|_{H^{-1}}^2 
&\leq \frac{1}{k}\sum_{m=1}^M \Bigl(\int^{t_m}_{t_{m-1}}(s-t_{m-1})^2\,ds\Bigr)
\Bigl(\int^{t_m}_{t_{m-1}}\|u_{tt}(s)\|_{H^{-1}}^2\,ds\Bigr)\\
&\leq Ck^2\rho_1(\eps,d).
\end{align*}

Introduce the error decompositions
\begin{align}\label{eq3.75}
E^m=\Theta^m+\Phi^m, \qquad G^m=\Lambda^m+\Psi^m,  
\end{align}
where 
\begin{alignat*}{2}
&\Theta^m:=u(t_m)-P_hu(t_m), &&\qquad \Phi^m:=P_hu(t_m)-U^m,  \\
&\Lambda^m:=w(t_m)-P_hw(t_m), &&\qquad \Psi^m:=P_hw(t_m)-W^m.  
\end{alignat*}
Using the definition of the operator $P_h$ in \eqref{eq3.11}, \eqref{eq3.73}--\eqref{eq3.74} 
can be rewritten as
\begin{align}\label{eq3.78}
&\bigl(d_t\Phi^m,\eta_h\bigr)+ a_h(\Psi^m,\eta_h) =-\bigl(d_t\Theta^m,\eta_h\bigr)
+\bigl(R(u_{tt},m),\eta_h \bigr) \qquad\forall\eta_h\in V_h,  \\
&\epsilon a_h(\Phi^m,v_h)+\frac{1}{\epsilon} \bigl(f(u(t_m))-f^m,v_h\bigr)
=\bigl(\Psi^m,v_h\bigr)+\bigl(\Lambda^m,v_h\bigr) \qquad\forall v_h\in V_h. \label{eq3.79}
\end{align}

Setting $\eta_h=-\Delta_h^{-1}\Phi^m$ in \eqref{eq3.78} and $v_h=\Phi^m$ in \eqref{eq3.79},
adding the resulting equations and summing over $m$ from $1$ to $\ell$, we get
\begin{align}
&a_h(\Delta_h^{-1}\Phi^\ell,\Delta_h^{-1}\Phi^\ell)
+\sum_{m=1}^\ell a_h(\Delta_h^{-1}\Phi^m-\Delta_h^{-1}\Phi^{m-1},\Delta_h^{-1}\Phi^m-\Delta_h^{-1}\Phi^{m-1})\label{eq3.80}\\
&\hskip 0.8in
+2k\sum_{m=1}^\ell \epsilon a_h(\Phi^m,\Phi^m)+2k\sum_{m=1}^\ell\frac{1}{\epsilon} \bigl(f(u(t_m))-f^m,\Phi^m\bigr)\nonumber\\
&\quad 
=2k\sum_{m=1}^\ell \Big( \bigl(R(u_{tt},m),-\Delta_h^{-1}\Phi^m\bigr)-\bigl(d_t\Theta^m,-\Delta_h^{-1}\Phi^m\bigr)
+\bigl(\Lambda^m,\Phi^m\bigr)\Big) \nonumber \\
&\hskip 0.8in
+a_h(\Delta_h^{-1}\Phi^0,\Delta_h^{-1}\Phi^0).\nonumber
\end{align}

\smallskip
{\em Step 2:} 
For $\sigma_e^0>\sigma_*^e$ for all $e\in \mathcal{E}_h$,  the first long term on the right-hand 
side of \eqref{eq3.80} can be bounded as follows
\begin{align} \label{eq3.81}
&2k\sum_{m=1}^\ell\Bigl( \bigl(R(u_{tt},m),-\Delta_h^{-1}\Phi^m\bigr)
+\bigl(d_t\Theta^m,-\Delta_h^{-1}\Phi^m\bigr)+\bigl(\Lambda^m,\Phi^m\bigr)\Bigr) \\
&\qquad
\leq Ck\sum_{m=1}^\ell \Bigl(\|R(u_{tt};m)\|_{H^{-1}}^2+\|d_t\Theta^m\|_{H^{-1}}^2
+(1-\eps^3)\epsilon^{-4}\|\Lambda^m\|_{H^{-1}}^2\Bigr)\nonumber\\
&\qquad\qquad
+k\sum_{m=1}^\ell\Bigl( a_h(\Delta_h^{-1}\Phi^m,\Delta_h^{-1}\Phi^m)
+\frac{\eps^4}{1-\eps^3} a_h(\Phi^m,\Phi^m)\Bigr)\nonumber\\
&\qquad
\leq k\sum_{m=1}^\ell\Bigl( a_h(\Delta_h^{-1}\Phi^m,\Delta_h^{-1}\Phi^m)
+\frac{\eps^4}{1-\eps^3} a_h(\Phi^m,\Phi^m)\Bigr)\nonumber\\
&\qquad\qquad
+C\Bigl(k^2\rho_1(\epsilon,d)+h^{6}\rho_3(\epsilon)\Bigr),\nonumber
\end{align}
where we have used \eqref{eq2.15} and the following facts \cite{Dupont}:
\begin{equation*}
\|u-P_hu\|_{H^{-1}}\leq Ch^3\|u\|_{H^2},\qquad 
\|w-P_hw\|_{H^{-1}}\leq Ch^3\|w\|_{H^2}.
\end{equation*}

We now bound the last term on the left-hand side of \eqref{eq3.80}. By the definition of $f^m$,
we have
\begin{align*}
&f(u(t_m))-f^m = f(u(t_m)) -f\bigl(P_h u(t_m)\bigr) + f\bigl(P_h u(t_m)\bigr)-f^m\\ 
&\qquad\qquad
\geq -\bigl|f(u(t_m)) -f\bigl(P_h u(t_m)\bigr) \bigr| 
+ \bigl(P_h u(t_m)\bigr)^3 -P_h u(t_m) - (U^m)^3 + U^{m-1} \nonumber \\
&\qquad\qquad
\geq -C|\Theta^m| + \Bigl( \bigl(P_h u(t_m)\bigr)^2 + P_h u(t_m)\,
U^m + (U^m)^2 \Bigr)\Phi^m -\Phi^m -kd_t U^m \nonumber \\
&\qquad\qquad
\geq -C|\Theta^m|+f'\bigl(P_h u(t_m)\bigr)\,\Phi^m-3P_h u(t_m)\,(\Phi^m)^2 
+(\Phi^m)^3-kd_t U^m. \nonumber
\end{align*}
By the discrete energy law \eqref{eq3.15}, \eqref{eq3.4} and \eqref{eq3.544}, we obtain for any $1\leq \ell\leq M$
\begin{align}\label{eq3.76}
&2k\sum_{m=1}^\ell \frac{1}{\eps}\bigl(f(u(t_m))-f^m,\Phi^m\bigr) \\
&\geq 
-\frac{Ck}{\eps}\sum_{m=1}^\ell \|\Theta^m\|_{H^{-1}(\cT_h)} \|\Phi^m\|_{H^1(\cT_h)}
+2k\sum_{m=1}^\ell \frac{1}{\eps}\Bigl( f'\bigl(P_h u(t_m)\bigr), (\Phi^m)^2\Bigr) \nonumber \\
&\quad 
-\frac{Ck}{\eps} \sum_{m=1}^\ell \|\Phi^m\|_{L^3}^3 
+\frac{2k}{\eps}\sum_{m=1}^\ell \|\Phi^m\|_{L^4}^4 
-\frac{2k}{\eps} \sum_{m=1}^\ell k\|d_t U^m\|_{-1,h}\,\|\Phi^m\|_{\alpha} \nonumber\\
&\geq 2k\sum_{m=1}^\ell \frac{1}{\eps}\Bigl( f'\bigl(P_h u(t_m)\bigr), (\Phi^m)^2 \Bigr) 
+\frac{2k}{\eps}\sum_{m=1}^\ell \|\Phi^m\|_{L^4}^4
-\frac{Ck}{\eps} \sum_{m=1}^\ell \|\Phi^m\|_{L^3}^3 \nonumber\\
&\quad 
-k\eps^4\sum_{m=1}^\ell \|\Phi^m\|_a^2 
-C\Bigl( h^{6} \eps^{-6} \|u\|_{L^2((0,T);H^s(\Ome))}^2
+k^2 \eps^{-6} E_h(u_h^0) \Bigr)   \nonumber \\
&\geq 2k\sum_{m=1}^\ell \frac{1}{\eps}\Bigl( f'\bigl(P_h u(t_m)\bigr), (\Phi^m)^2 \Bigr) 
+\frac{2k}{\eps}\sum_{m=1}^\ell \|\Phi^m\|_{L^4}^4
-\frac{Ck}{\eps} \sum_{m=1}^\ell \|\Phi^m\|_{L^3}^3\notag\\
&\quad
-k\frac{\eps^4}{1-\eps^3}\sum_{m=1}^\ell a_h(\Phi^m,\Phi^m)
-C\Bigl( h^{6} \eps^{-6} \|u\|_{L^2((0,T);H^s(\Ome))}^2 
+k^2 \eps^{-6} E_h(U^0) \Bigr).\nonumber
\end{align}

Substituting \eqref{eq3.81} and \eqref{eq3.76} into \eqref{eq3.80} we get
\begin{align}
&a_h(\Delta_h^{-1}\Phi^\ell,\Delta_h^{-1}\Phi^\ell)
+\sum_{m=1}^\ell a_h(\Delta_h^{-1}\Phi^m-\Delta_h^{-1}\Phi^{m-1},\Delta_h^{-1}\Phi^m-\Delta_h^{-1}\Phi^{m-1})
\label{eq3.84}\\
&\hskip 0.6in 
+\frac{2k(1-5\eps^3)}{1-\eps^3}\sum_{m=1}^\ell
\Bigl(\eps a_h(\Phi^m,\Phi^m)+\frac{1-\eps^3}{\eps} \bigl(f'(P_hu(t_m))\Phi^m,\Phi^m\bigr)\Bigr)\notag \\
& \hskip 0.6in
+\frac{6\eps^4}{1-\eps^3}k \sum_{m=1}^\ell a_h(\Phi^m,\Phi^m) 
+\frac{2k}{\eps}\sum_{m=1}^\ell \|\Phi^m\|_{L^4}^4 \nonumber \\
&\hskip 0.3in
\leq Ck\sum_{m=1}^\ell a_h(\Delta_h^{-1}\Phi^m,\Delta_h^{-1}\Phi^m)
+\frac{Ck}{\eps}\sum_{m=1}^\ell\|\Phi^m\|_{L^{3}}^{3} \notag\\
& \hskip 0.6in
-10k\eps^2\sum_{m=1}^\ell \bigl(f'(P_hu(t_m))\Phi^m,\Phi^m \bigr)
+C\bigr(k^2\rho_1(\eps;d)+h^6\rho_3(\eps)\bigr)  \nonumber \\
&\hskip 0.6in
+C\Bigl( h^{6} \eps^{-6} \|u\|_{L^2((0,T);H^s(\Ome)}^2 +k^2 \eps^{-6} E_h(U^0) \Bigr).\nonumber
\end{align}

\smallskip
{\em Step 3:} 
To control the second term on the right-hand side of \eqref{eq3.84}, we appeal to the following 
Gagliardo-Nirenberg inequality \cite{Adams03}:
\[
\|v\|_{L^3(K)}^3\leq C\Bigl( \|\nab v\|_{L^2(K)}^{\frac{d}2} 
\bigl\|v\bigr\|_{L^2(K)}^{\frac{6-d}2} +\|v\|_{L^2(K)}^3 \Bigr)
\qquad\forall K\in \cT_h.
\]
Thus we get
\begin{align}\label{eq3.87b}
\frac{Ck}{\eps}\sum_{m=1}^\ell \|\Phi^m\|_{L^3}^3   
\leq& \eps^4 k\sum_{m=1}^\ell \|\nab \Phi^m\|_{L^2(\cT_h)}^2
+\frac{Ck}{\eps}\sum_{m=1}^\ell \|\Phi^m\|_{L^2}^3  \\
&\hskip 0.3in
+C\eps^{-\frac{4(1+d)}{4-d}} k\sum_{m=1}^\ell
\bigl\|\Phi^m\bigr\|_{L^2}^{\frac{2(6-d)}{4-d}} \nonumber \\
\leq &\frac{\eps^4}{1-\eps^3}k\sum_{m=1}^\ell a_h(\Phi^m,\Phi^m)+
\frac{Ck}{\eps}\sum_{m=1}^\ell \|\Phi^m\|_{L^2}^3 \nonumber \\
&\hskip 0.3in
+C\eps^{-\frac{4(1+d)}{4-d}} k\sum_{m=1}^\ell \bigl\|\Phi^m\bigr\|_{L^2}^{\frac{2(6-d)}{4-d}}. 
\nonumber
\end{align}

The third item on the right-hand side of \eqref{eq3.84} can be bounded by
\begin{align} \label{eq3.87}
&-10k\eps^2(f'(P_hu(t_m))\Phi^m,\Phi^m) \\
&\hskip 0.8in
\leq k\frac{\eps^4}{1-\eps^3} a_h(\Phi^m,\Phi^m)
+kC a_h(\Delta_h^{-1}\Phi^m,\Delta_h^{-1}\Phi^m). \nonumber
\end{align}
Again, here we have used \eqref{eq3.544}.

Finally, for the third term on the left-hand side of \eqref{eq3.84}, 
we utilize the discrete spectrum estimate \eqref{eq2.24} to bound it from below as follows:
\begin{align} \label{sec3:add3}
\epsilon a_h(\Phi^m,\Phi^m)+\frac{1-\eps^3}{\eps} \bigl(f'(P_hu(t_m))\Phi^m,\Phi^m\bigr) 
\geq -c_0 \|\nabla\Delta^{-1}\Phi^m\|_{L^2}^2.
\end{align}
By the stability of $\Delta^{-1}$ and \eqref{eq3.544}, we also have
\begin{align} \label{sec3:spec7}
c_0 \|\nabla\Delta^{-1}\Phi^m\|_{L^2}^2 &\leq C \|\Phi^m\|_{L^2}^2 
\leq \frac{\eps^4}{1-\eps^3}a_h(\Phi^m,\Phi^m)+Ca_h(\Delta_h^{-1}\Phi^m,\Delta_h^{-1}\Phi^m).  
\end{align}

\smallskip
{\em Step 4:} 
Substituting \eqref{eq3.87b}, \eqref{eq3.87}, \eqref{sec3:add3}, \eqref{sec3:spec7} 
into \eqref{eq3.84}, we get
\begin{align} \label{eq3.87c}
& a_h(\Delta_h^{-1}\Phi^\ell,\Delta_h^{-1}\Phi^\ell)
+\sum_{m=1}^\ell\limits 
a_h(\Delta_h^{-1}\Phi^m-\Delta_h^{-1}\Phi^{m-1},\Delta_h^{-1}\Phi^m-\Delta_h^{-1}\Phi^{m-1}) \\
&\hskip 1.2in
+\frac{2\eps^4k}{1-\eps^3}\sum_{m=1}^\ell a_h(\Phi^m,\Phi^m)
+\frac{2k}{\eps}\sum_{m=1}^\ell \|\Phi^m\|_{L^4}^4\nonumber\\
&\quad
\leq Ck\sum_{m=1}^\ell a_h(\Delta^{-1}_h\Phi^m,\Delta^{-1}_h\Phi^m) 
+\frac{Ck}{\eps}\sum_{m=1}^\ell \|\Phi^m\|_{L^2}^3 \nonumber \\
&\qquad	 
+C\eps^{-\frac{4(1+d)}{4-d}} k\sum_{m=1}^\ell \bigl\|\Phi^m\bigr\|_{L^2}^{\frac{2(6-d)}{4-d}}
+C\bigl( k^2\rho_1(\eps;d)+h^6\rho_3(\eps) \bigr) \nonumber \\
&\qquad
+C\Bigl( h^{6} \eps^{-6} \|u\|_{L^2((0,T);H^s(\Ome)}^2 +k^2 \eps^{-6} E_h(U^0) \Bigr).\nonumber
\end{align}

By discrete energy law \eqref{eq3.15}, General Assumption \eqref{eq2.3}, $H^1$ stability of elliptic projection, $L^{\infty}$ stability(or $L^{\infty}$ error estimate and triangle inequality) of elliptic projection, we can get for any $0\leq \ell\leq M$
\begin{align*}
\|U^\ell\|_{L^2} \leq k\sum\limits_{m=1}^\ell \|d_t U^m\|_{L^2} +\|U^0\|_{L^2} \leq C\eps^{-\sigma_1}.
\end{align*}
Since the projection of $u$ is bounded, then for any $0\leq \ell\leq M$
\begin{equation}\label{20150201_4}
\|\Phi^\ell\|_{L^2}\leq C\eps^{-\sigma_1}.
\end{equation}
We point out that the exponent for $\bigl\|\Phi^m\bigr\|_{L^2}$ is $\frac{2(6-d)}{4-d}$, which is bigger
than $3$ for $d=2,3$. By \eqref{20150201_4} we have
\begin{align*}
\bigl\|\Phi^m\bigr\|_{L^2}^{4}\leq C\eps^{-\sigma_1}\bigl\|\Phi^m\bigr\|_{L^2}^{3},\qquad
\bigl\|\Phi^m\bigr\|_{L^2}^{6}\leq C\eps^{-3\sigma_1}\bigl\|\Phi^m\bigr\|_{L^2}^{3}.
\end{align*}

Using the Schwarz and Young's inequalities, we have
\begin{align}\label{eq3.87h}
\bigl\|\Phi^m\bigr\|_{L^2}^{3}
&=\Bigl(\bigl\|\Phi^m\bigr\|_{L^2}^{2}\Bigr)^{\frac{3}{2}}
=a_h(-\Delta_h^{-1}\Phi^m,\Phi^m)^{\frac{3}{2}} \\
&\leq a_h(\Delta_h^{-1}\Phi^m,\Delta_h^{-1}\Phi^m)^{\frac{3}{4}}\,
a_h(\Phi^m,\Phi^m)^{\frac{3}{4}} \notag \\
&\leq \eps^{\frac{4(1+d)}{4-d}+\sigma_1+2(d-2)\sigma_1} \frac{\eps^4}{1-\eps^3} a_h(\Phi^m,\Phi^m)\notag\\
&\qquad+C\eps^{-4}\eps^{-\frac{4(1+d)}{4-d}-\sigma_1-2(d-2)\sigma_1} a_h(\Delta_h^{-1}\Phi^m,\Delta_h^{-1}\Phi^m)^3. \notag
\end{align}
Therefore, \eqref{eq3.87c} becomes 
\begin{align}\label{eq3.87i}
& a_h(\Delta_h^{-1}\Phi^\ell,\Delta_h^{-1}\Phi^\ell)
+\sum_{m=1}^\ell a_h(\Delta_h^{-1}\Phi^m-\Delta_h^{-1}\Phi^{m-1},\Delta_h^{-1}\Phi^m-\Delta_h^{-1}\Phi^{m-1}) \\
&\hskip 1.2in
+\frac{\eps^4k}{1-\eps^3}\sum_{m=1}^\ell a_h(\Phi^m,\Phi^m)+\frac{2k}{\eps}
\sum_{m=1}^\ell \|\Phi^m\|_{L^4}^4\nonumber\\
&\quad
\leq Ck\sum_{m=1}^\ell a_h(\Delta_h^{-1}\Phi^m,\Delta_h^{-1}\Phi^m)\nonumber \\
&\qquad +Ck\eps^{-\frac{4(6+d)}{4-d}-2\sigma_1-4(d-2)\sigma_1}\sum_{m=1}^\ell a_h(\Delta_h^{-1}\Phi^m,\Delta_h^{-1}\Phi^m)^{3}\nonumber\\
&\qquad
+C\bigl( k^2\rho_1(\eps;d)+h^6\rho_3(\eps) \bigr)
+C\Bigl( h^{6} \eps^{-6} \|u\|_{L^2((0,T);H^s(\Ome)}^2 +k^2 \eps^{-6} E_h(U^0) \Bigr)\nonumber\\
&\quad
\leq Ck\sum_{m=1}^\ell a_h(\Delta_h^{-1}\Phi^m,\Delta_h^{-1}\Phi^m)\nonumber \\
&\qquad +Ck\eps^{-\frac{4(6+d)}{4-d}-2\sigma_1-4(d-2)\sigma_1}\sum_{m=1}^\ell a_h(\Delta_h^{-1}\Phi^m,\Delta_h^{-1}\Phi^m)^{3}\nonumber\\
&\qquad
+C\bigl( k^2\rho_1(\eps;d)+h^6\rho_3(\eps) \bigr).\nonumber
\end{align}

On noting that $U^m$ can be written as
\begin{align}\label{20150207_6}
U^{\ell} = k\sum_{m=1}^\ell d_t U^m+U^0,
\end{align}
then by \eqref{eq2.3} and \eqref{eq3.15}, we get
\begin{align}\label{20150207_7}
\|U^{\ell}\|_{-1,h} \leq k\sum_{m=1}^\ell\|d_t U^m\|_{-1,h}+\|U^0\|_{-1,h}\leq C\eps^{-\sigma_1}.
\end{align}
Using the boundedness of the projection, we have
\begin{align}\label{20150207_8}
\|\Phi^{\ell}\|_{-1,h}^2\leq C\eps^{-2\sigma_1}.
\end{align}

Also, \eqref{eq3.87i} can be written in the following equivalent form
\begin{align}\label{20150207_9}
& a_h(\Delta_h^{-1}\Phi^\ell,\Delta_h^{-1}\Phi^\ell)
+\sum_{m=1}^\ell a_h(\Delta_h^{-1}\Phi^m-\Delta_h^{-1}\Phi^{m-1},\Delta_h^{-1}\Phi^m-\Delta_h^{-1}\Phi^{m-1}) \\
&\hskip 0.9in
+\frac{\eps^4k}{1-\eps^3}\sum_{m=1}^\ell a_h(\Phi^m,\Phi^m)+\frac{2k}{\eps}
\sum_{m=1}^\ell \|\Phi^m\|_{L^4}^4\leq M_1+M_2,\nonumber
\end{align}
where
\begin{align}\label{20150207_10}
M_1 &:= Ck\sum_{m=1}^{\ell-1} a_h(\Delta_h^{-1}\Phi^m,\Delta_h^{-1}\Phi^m)\\
&\qquad
+Ck\eps^{-\frac{4(6+d)}{4-d}-2\sigma_1-4(d-2)\sigma_1}\sum_{m=1}^{\ell-1} a_h(\Delta_h^{-1}\Phi^m,\Delta_h^{-1}\Phi^m)^{3}\nonumber\\
&\qquad +C\bigl( k^2\rho_1(\eps;d)+h^6\rho_3(\eps) \bigr),\nonumber\\
M_2 &:=Ck a_h(\Delta_h^{-1}\Phi^\ell,\Delta_h^{-1}\Phi^\ell)
\label{20150207_11}\\
&\qquad+Ck\eps^{-\frac{4(6+d)}{4-d}-2\sigma_1-4(d-2)\sigma_1}a_h(\Delta_h^{-1}\Phi^{\ell},\Delta_h^{-1}\Phi^{\ell})^{3}.\nonumber
\end{align}
It is easy to check that
\begin{align}\label{20150207_12}
M_2\leq\frac12\|\Phi^{\ell}\|_{-1,h}^2\qquad\text{provided that}\quad k\leq C\eps^{\frac{4(6+d)}{4-d}+(4d-2)\sigma_1}.
\end{align}
Under this restriction, we have
\begin{align}\label{20150207_13}
& a_h(\Delta_h^{-1}\Phi^\ell,\Delta_h^{-1}\Phi^\ell)
+2\sum_{m=1}^\ell a_h(\Delta_h^{-1}\Phi^m-\Delta_h^{-1}\Phi^{m-1},\Delta_h^{-1}\Phi^m-\Delta_h^{-1}\Phi^{m-1}) \\
&\hskip 1.2in
+\frac{2\eps^4k}{1-\eps^3}\sum_{m=1}^\ell a_h(\Phi^m,\Phi^m)+\frac{4k}{\eps}
\sum_{m=1}^\ell \|\Phi^m\|_{L^4}^4\nonumber\\
&\quad
\leq 2Ck\sum_{m=1}^{\ell-1} a_h(\Delta_h^{-1}\Phi^m,\Delta_h^{-1}\Phi^m)+2C\bigl( k^2\rho_1(\eps;d)+h^6\rho_3(\eps) \bigr)\nonumber \\
&\qquad
+2Ck\eps^{-\frac{4(6+d)}{4-d}-2\sigma_1-4(d-2)\sigma_1}\sum_{m=1}^{\ell-1} a_h(\Delta_h^{-1}\Phi^m,\Delta_h^{-1}\Phi^m)^{3}\nonumber\\
&\quad
\leq Ck\sum_{m=1}^{\ell-1} a_h(\Delta_h^{-1}\Phi^m,\Delta_h^{-1}\Phi^m)+C\bigl( k^2\rho_1(\eps;d)+h^6\rho_3(\eps) \bigr)\nonumber \\
&\qquad
+Ck\eps^{-\frac{4(6+d)}{4-d}-2\sigma_1-4(d-2)\sigma_1}\sum_{m=1}^{\ell-1} a_h(\Delta_h^{-1}\Phi^m,\Delta_h^{-1}\Phi^m)^{3}.
\nonumber
\end{align}

Define the slack variable $d_\ell\geq 0$ such that 
\begin{align}\label{eq3.87j}
&a_h(\Delta_h^{-1}\Phi^\ell,\Delta_h^{-1}\Phi^\ell)
+2\sum_{m=1}^\ell a_h(\Delta_h^{-1}\Phi^m-\Delta_h^{-1}\Phi^{m-1},\Delta_h^{-1}\Phi^m-\Delta_h^{-1}\Phi^{m-1}) \\
&\hskip 1.2in
+\frac{2\eps^4k}{1-\eps^3}\sum_{m=1}^\ell a_h(\Phi^m,\Phi^m)+\frac{4k}{\eps}
\sum_{m=1}^\ell \|\Phi^m\|_{L^4)}^4 +d_\ell\nonumber\\
&\quad
=Ck\sum_{m=1}^{\ell-1} a_h(\Delta_h^{-1}\Phi^m,\Delta_h^{-1}\Phi^m)+C\bigl(k^2\rho_1(\eps;d)+h^6\rho_3(\eps)\bigr) \nonumber \\
&\qquad
+Ck\eps^{-\frac{4(6+d)}{4-d}-2\sigma_1-4(d-2)\sigma_1}\sum_{m=1}^{\ell-1} a_h(\Delta_h^{-1}\Phi^m,\Delta_h^{-1}\Phi^m)^{3}.\nonumber
\end{align}

We also define $\{S_{\ell}\}_{\ell\geq 1}$ by
\begin{align}
&S_{\ell} = d_\ell
+2\sum_{m=1}^\ell a_h(\Delta_h^{-1}\Phi^m-\Delta_h^{-1}\Phi^{m-1},\Delta_h^{-1}\Phi^m-\Delta_h^{-1}\Phi^{m-1})\label{eq3.87l}\\
&\hskip 0.5in
+a_h(\Delta_h^{-1}\Phi^\ell,\Delta_h^{-1}\Phi^\ell)
+\frac{2\eps^4k}{1-\eps^3}\sum_{m=1}^\ell a_h(\Phi^m,\Phi^m)
+\frac{4k}{\eps}\sum_{m=1}^\ell \|\Phi^m\|_{L^4}^4, \nonumber
\end{align}
and equation \eqref{eq3.87j} shows that 
\begin{align*} 
&S_1 = C\bigl(k^2\rho_1(\eps;d)+h^6\rho_3(\eps)\bigr).
\end{align*}
Then
\begin{equation}\label{eq3.87m}
S_{\ell+1}-S_{\ell}\leq CkS_{\ell}+Ck\eps^{-\frac{4(6+d)}{4-d}-2\sigma_1-4(d-2)\sigma_1}S_{\ell}^3 \qquad\forall \ell\geq 1.
\end{equation}

Applying Lemma \ref{lem2.3} to $\{S_{\ell}\}_{\ell\geq1}$ defined above, we obtain $\forall\ell\geq1$,
\begin{equation}\label{20150207_14}
S_{\ell}\leq a^{-1}_{\ell}\Bigg\{S^{-2}_{1}-2C\eps^{-\frac{4(6+d)}{4-d}-2\sigma_1-4(d-2)\sigma_1}k\sum_{s=1}^{\ell-1} a^{-2}_{s+1}\Bigg\}^{-\frac{1}{2}}
\end{equation}
provided that
\begin{equation}\label{20150207_15}
S^{-2}_{1}-2C\eps^{-\frac{4(6+d)}{4-d}-2\sigma_1-4(d-2)\sigma_1}k\sum_{s=1}^{\ell-1} a^{-2}_{s+1}>0.
\end{equation}
We note that $a_s\, (1\leq s\leq \ell)$ are all bounded as $k\rightarrow0$, 
therefore, \eqref{20150207_15} holds under the mesh constraint stated in the theorem.  
It follows from \eqref{eq3.58} and \eqref{eq3.59} that 
\begin{equation}\label{20150207_16}
S_{\ell}\leq 2a_\ell^{-1} S_1
\leq C\bigl(k^2\rho_1(\eps;d)+h^6\rho_3(\eps)\bigr).
\end{equation}

Then \eqref{eq3.60} follows from the triangle inequality 
on $E^m=\Theta^m+\Phi^m$. \eqref{eq3.60b} is obtained by taking the test function $\eta_h=\Phi^m$ in \eqref{eq3.78} and $v_h=\Phi^m$ in \eqref{eq3.79}, and \eqref{eq3.60a} is 
a consequence of the Poincar\`e inequality.  

Now setting $\eta_h=\Phi^m$ in \eqref{eq3.78} and $v_h=-\frac{1}{\eps}\Psi^m$ in \eqref{eq3.79}, 
and adding the resulting equations yield
\begin{align} \label{eq3.91}
&\frac{1}{2} d_t\|\Phi^m\|_{L^2}^2+\frac{k}{2}\|d_t\Phi^m\|_{L^2}^2 +\frac{1}{\eps}\|\Psi^m\|_{L^2}^2 
=\frac{1}{\eps^2}\bigl(f(u(t_m))-f(U^m),\Psi^m\bigr) \\
&\hskip 1.6in
+\bigl(R(u_{tt};m),\Phi^m\bigr)-\bigl(d_t\Theta^m,\Phi^m\bigr)
-\frac{1}{\eps}\bigl(\Lambda^m,\Psi^m\bigr). \nonumber 
\end{align}
The last three terms on the right-hand side of \eqref{eq3.91} can be bounded in
the same way as in \eqref{eq3.81}, and the first term can be controlled as
\begin{align}\label{eq3.92}
\frac{1}{\eps^2}\bigl(f(u(t_m))-f(U^m),\Psi^m\bigr)&=\frac{1}{\eps^2}\bigl(f'(\xi)E^m,\Psi^m\bigr)\\
&\leq \frac{1}{2\eps}\|\Psi^m\|_{L^2}^2+\frac{C}{\eps^3}\|E^m\|_{L^2}^2. \notag
\end{align}

Multiplying both sides of \eqref{eq3.91} by $k$ and summing over $m$ from $1$ to $M$ yield
the desired estimate \eqref{eq3.61a}. Estimate \eqref{eq3.62} follows from an applications of
the following inverse inequality: 
\begin{equation}\label{eq3.92a}
\|\Phi^m\|_{L^{\infty}} \leq h^{-\frac{d}{2}}\|\Phi^m\|_{L^{2}},
\end{equation}
and the following $L^{\infty}$ estimate for the elliptic projection:
\begin{equation} \label{eq3.92b}
\|u-P_hu\|_{L^{\infty}}\leq Ch^{2}|\ln h|\|u\|_{W^{s,\infty}}\qquad \forall u\in H^2(\Omega).
\end{equation}

Finally, it is well known that there holds the following estimate for
the elliptic projection operator:
\begin{equation}\label{eq3.93}
\max_{0\leq m\leq M}\|\Lambda^m\|_{L^2}+\bigg(k\sum_{m=0}^M k\|d_t\Lambda^m\|_{L^2}^2\bigg)^{\frac{1}{2}}
\leq Ch^2\rho_2(\eps).
\end{equation}
Using the identity
\begin{equation}\label{eq3.94}
\bigl(d_t\Phi^m,\Phi^m\bigr)=\frac{1}{2}d_t\|\Phi^m\|_{L^2}^2+\frac{k}{2}\|d_t\Phi^m\|_{L^2}^2,
\end{equation}
we get
\begin{align} \label{eq3.95}
&\frac{1}{2}\|\Psi^M\|_{L^2}^2+k\sum_{m=1}^M \frac{k}{2}\|d_t\Psi^m\|_{L^2}^2
=k\sum_{m=1}^M \bigl(d_t\Psi^m,\Psi^m\bigr) +\frac{1}{2}\|\Psi^0\|_{L^2}^2 \\
&\qquad\qquad
\leq k\sum_{m=1}^M \bigg(\frac{k}{4}\|d_t\Psi^m\|_{L^2}^2
+\frac{1}{k}\|\Psi^m\|_{L^2}^2\bigg)+\frac{1}{2}\|\Psi^0\|_{L^2}^2.\nonumber
\end{align}
The first term on the right hand side of \eqref{eq3.95} can be absorbed by the second term on the left hand side of \eqref{eq3.95}. The second tern on the right hand side of \eqref{eq3.95} has been obtained in \eqref{eq3.61a}. Estimate \eqref{eq3.63a} for $W^m$ then follows from 
\eqref{eq3.93} and \eqref{eq3.95}. \eqref{eq3.63b} follows from an application of
the triangle inequality, the inverse inequality, and \eqref{eq3.92b}. This completes the proof.
\end{proof}

\section{Convergence of numerical interfaces}\label{sec-4}

In this section, we prove that the numerical interface defined 
as the zero level set of the finite element interpolation of the solution $U^m$ converges
to the moving interface of the Hele-Shaw problem under the assumption that the Hele-Shaw 
problem has a unique global (in time) classical solution. To the end, we first cite the following 
PDE convergence result proved in \cite{Alikakos94}.

\begin{theorem}\label{thm4.1}
Let $\Omega$ be a given smooth domain and $\Gamma_{00}$ be a smooth closed hypersurface 
in $\Omega$. Suppose that the Hele-Shaw problem starting from $\Gamma_{00}$ has a unique
smooth solution $\bigl(w,\Gamma:=\bigcup_{0\leq t\leq T}(\Gamma_t\times\{t\}) \bigr)$ 
in the time interval $[0,T]$ such that $\Gamma_t\subseteq\Omega$ for all $t\in[0,T]$. 
Then there exists a family of smooth functions $\{u_{0}^{\epsilon}\}_{0<\epsilon\leq 1}$ 
which are uniformly bounded in $\epsilon\in(0,1]$ and $(x,t)\in \overline{\Omega}_T$, 
such that if $u^{\epsilon}$ solves the Cahn-Hilliard problem \eqref{eq1.1}--\eqref{eq1.4}, then
\begin{itemize}
\item[\rm (i)] $\displaystyle{\lim_{\epsilon\rightarrow 0}} 
u^{\eps}(x,t)= \begin{cases}
1 &\qquad \mbox{if}\, (x,t)\in \mathcal{O}\\
-1 &\qquad \mbox{if}\, (x,t)\in \mathcal{I}
\end{cases}
\,\mbox{ uniformly on compact subsets}$,
where $\mathcal{I}$ and $\mathcal{O}$ stand for the ``inside" and ``outside" of $\Gamma$;
\item[\rm (ii)] $\displaystyle{\lim_{\epsilon\rightarrow 0}}
\bigl( \epsilon^{-1} f(u^{\epsilon})-\epsilon\Delta u^{\epsilon} \bigr)(x,t)=-w(x,t)$ uniformly on 
$\overline{\Omega}_T$.
\end{itemize}
\end{theorem}

\smallskip
We note that since $U^m$ is multi-valued on the edges of the mesh $\cT_h$, its zero-level
set is not well defined. To avoid this technicality, we use a continuous finite element 
interpolation of $U^m$ to define the numerical interface.  Let $\widehat{U}^m\in S_h$
denote the finite element approximation of $U^m$ which is defined using the averaged degrees 
of freedom of $U^m$ as the degrees of freedom for determining $\widehat{U}^m$ (cf. \cite{KP04}).
The following approximation results were proved in Theorem 2.1 of \cite{KP04}.

\begin{theorem}\label{lem4.1} 
Let $\mathcal{T}_h$ be a conforming mesh consisting of 
triangles when $d=2$, and tetrahedra when $d=3$. For $v_h\in V_h$, let 
$\widehat{v}_h$ be the finite element approximation  of $v_h$ as
defined above. Then for any $v_h\in V_h$ and $i=0,1$ there holds
\begin{align}\label{eqn_KP}
\sum_{K\in\mathcal{T}_h} \|v_h-\widehat{v}_h\|_{H^i(K)}^2
\leq C \sum_{e\in\cE_h^I} h^{1-2i}_e \|[v_h]\|_{L^2(e)}^2,
\end{align}
where $C>0$ is a constant independent of $h$ and $v_h$ but may depend on $r$ and the minimal 
angle $\theta_0$ of the triangles in $\mathcal{T}_h$. 
\end{theorem}

By the construction, $\widehat{U}^m$ is expected to be very close to $U^m$, hence, 
$\widehat{U}^m$ should also be very close to $u(t_m)$. This is indeed the case as
stated in the following theorem, which says that Theorem \ref{thm3.1} also hold for $\widehat{U}^m$.

\begin{theorem}\label{lem4.2}
Let $U^{m}$ denote the solution of scheme \eqref{eq3.1}--\eqref{eq3.5}
and $\widehat{U}^{m}$ denote its finite element approximation as defined above. Then 
under the assumptions of Theorem \ref{thm3.1} the error estimates for $U^m$ given in 
Theorem \ref{thm3.1} are still valid for $\widehat{U}^{m}$, in particular, there holds
\begin{align}\label{eq3.36bx}
\max_{0\leq m\leq M} \|u(t_m)-\widehat{U}^m\|_{L^\infty(\cT_h)}
&\leq C\Bigl(h^{2}|\ln h|\epsilon^{-\gamma} 
+h^{-\frac{d}{2}}\epsilon^{-\frac{7}{2}} r(h,k;\epsilon,d,\sigma_i)^{\frac{1}{2}} \Bigr).
\end{align}

\end{theorem}

We omit the proof to save space and refer the reader to \cite{Feng_Li14} to see a proof
of the same nature for the related Allen-Cahn problem. 

%

We are now ready to state the first main theorem of this section.

\begin{theorem}\label{thm4.2}
Let $\{\Gamma_t\}_{t\geq0}$ denote the zero level set of the Hele-Shaw problem 
and  $\bigl( U_{\epsilon,h,k}(x,t), W_{\epsilon,h,k}(x,t)\bigr)$ denote the piecewise linear 
interpolation in time of the finite element interpolation $\{(\widehat{U}^m, \widehat{W}^m)\}$ of the DG 
solution $\{(U^m, W^m)\}$, namely,
\begin{align}
&U_{\epsilon,h,k}(x,t):=\frac{t-t_{m-1}}{k}\widehat{U}^{m}(x)+\frac{t_{m}-t}{k}\widehat{U}^{m-1}(x),	\label{eq4.1}\\
&W_{\epsilon,h,k}(x,t):=\frac{t-t_{m-1}}{k}W^{m}(x)+\frac{t_{m}-t}{k}W^{m-1}(x),	\label{eq4.2}
\end{align}
for $t_{m-1}\leq t\leq t_{m}$ and $1\leq m\leq M$.
Then, under the mesh and starting value constraints of Theorem \ref{thm3.1} 
and $k=O(h^{2-\gamma})$ with $\gamma>0$, we have
\begin{itemize}
\item[\rm (i)] $U_{\epsilon,h,k}(x,t) \stackrel{\eps\searrow 0}{\longrightarrow} 1$ uniformly 
on compact subset of $\mathcal{O}$,
\item[\rm (ii)] $U_{\epsilon,h,k}(x,t) \stackrel{\eps\searrow 0}{\longrightarrow} -1$ uniformly on 
compact subset of $\mathcal{I}$. 
\item[\rm (iii)] Moreover, in the case that dimension $d=2$, when $k=O(h^{3})$, 
suppose that $W^0$ satisfies $\|w_0^{\epsilon}-W^0\|_{L^2}\leq Ch^{\beta}$ for some 
$\beta>\frac32$, then we have
$W_{\epsilon,h,k}(x,t)\stackrel{\eps\searrow 0}{\longrightarrow} -w(x,t)$ uniformly on $\overline{\Omega}_T$.
\end{itemize}
\end{theorem}

\smallskip
\begin{proof}
For any compact set $A\subset\mathcal{O}$ and for any $(x,t)\in A$, we have
\begin{align} \label{eq4.4}
|U_{\epsilon,h,k}-1|&\leq |U_{\epsilon,h,k}-u^{\epsilon}(x,t)|+|u^{\epsilon}(x,t)-1| \\
&\leq |U_{\epsilon,h,k}-u^{\epsilon}(x,t)|_{L^{\infty}(\Omega_T)}+|u^{\epsilon}(x,t)-1|.\nonumber
\end{align}
Equation \eqref{eq3.62} of Theorem \ref{thm3.1} infers that there exists a constant $0<\alpha<\frac{4-d}{2}$ such that
\begin{equation}\label{eq4.5}
|U_{\epsilon,h,k}-u^{\epsilon}(x,t)|_{L^{\infty}(\Omega_T)}\leq Ch^{\alpha}.
\end{equation}
The first term on the right-hand side of \eqref{eq4.4} tends to $0$ when $\epsilon\searrow 0$
(note that $h,k\searrow 0$, too). The second term converges uniformly to $0$ on the compact set $A$, 
which is ensured by (i) of Theorem \ref{thm4.1}.  Hence, the assertion (i) holds.

To show (ii), we only need to replace $\mathcal{O}$ by $\mathcal{I}$ and $1$ by $-1$ in the above proof.
To prove (iii), under the assumptions $k=O(h^{3})$, 
\eqref{eq3.63b} in Theorem \ref{thm3.1} implies that there exists a positive constant $0<\zeta<\frac{4-d}{2}$ 
such that
\begin{equation}\label{eq4.6}
\|W_{\epsilon,h,k}-w^{\epsilon}\|_{L^{\infty}(\Omega_T)}\leq Ch^{\zeta}.
\end{equation}
Then by the triangle inequality we obtain for any $(x,t)\in \overline{\Omega}_T$,
\begin{align} \label{eq4.7}
|W_{\epsilon,h,k}(x,t)-(-w)|&\leq|W_{\epsilon,h,k}(x,t)-w^{\epsilon}(x,t)|+|w^{\epsilon}(x,t)-(-w)|,\\
&\leq\|W_{\epsilon,h,k}(x,t)-w^{\epsilon}(x,t)\|_{L^{\infty}(\Omega_T)}+|w^{\epsilon}(x,t)-(-w)|.\nonumber
\end{align}
The first term on the right-hand side of \eqref{eq4.7} tends to $0$ when $\epsilon\searrow 0$
(note that $h,k\searrow 0$, too). The second term converges uniformly to $0$ in $\overline{\Omega}_T$, 
which is ensured by (ii) of Theorem \ref{thm4.1}. Thus the assertion (iii) is proved. The proof is complete.
\end{proof}

The second main theorem of this section which is given below addresses the convergence 
of numerical interfaces.

\begin{theorem}\label{thm4.3}
Let $\Gamma_t^{\epsilon,h,k}:=\{x\in\Omega;\, U_{\epsilon,h,k}(x,t)=0\}$
be the zero level set of $U_{\epsilon,h,k}(x,t)$, then under the 
assumptions of Theorem \ref{thm4.2}, we have
\[
\sup_{x\in\Gamma_t^{\epsilon,h,k}} \mbox{\rm dist}(x,\Gamma_t)
\stackrel{\epsilon\searrow 0}{\longrightarrow} 0 \quad\mbox{uniformly on $[0,T]$}.
\]
\end{theorem}

\begin{proof}
For any $\eta\in(0,1)$,  define the open tabular neighborhood $\mathcal{N}_{\eta}$ of width 
$2\eta$ of $\Gamma_t$ as
\begin{equation}\label{eq4.8}
\mathcal{N}_{\eta}:=\{(x,t)\in\Omega_T;\, \mbox{\rm dist}(x,\Gamma_t)<\eta\}.
\end{equation}
Let $A$ and $B$ denote the complements of the neighborhood $\mathcal{N}_{\eta}$ in $\mathcal{O}$ 
and $\mathcal{I}$, respectively, i.e.
\begin{equation*}
A=\mathcal{O}\setminus\mathcal{N}_{\eta} \qquad\mbox{and}\qquad
B=\mathcal{I}\setminus\mathcal{N}_{\eta}.
\end{equation*}
Note that $A$ is a compact subset outside $\Gamma_t$ and $B$ is a compact subset inside $\Gamma_t$,
then there exists ${\epsilon_3}>0$, which only depends on $\eta$, such that 
for any $\epsilon\in (0,{\epsilon_3})$
\begin{align}
&|U_{\epsilon,h,k}(x,t)-1|\leq\eta\quad\forall(x,t)\in A,\label{eq4.9}\\
&|U_{\epsilon,h,k}(x,t)+1|\leq\eta\quad\forall(x,t)\in B.\label{eq4.10}
\end{align}
Now for any $t\in[0,T]$ and $x\in \Gamma_t^{\epsilon,h,k}$, from $U_{\epsilon,h,k}(x,t)=0$ we have
\begin{align}
&|U_{\epsilon,h,k}(x,t)-1|=1\qquad\forall(x,t)\in A,\label{eq4.11}\\
&|U_{\epsilon,h,k}(x,t)+1|=1\qquad\forall(x,t)\in B.\label{eq4.12}
\end{align}
\eqref{eq4.9} and \eqref{eq4.11} imply that $(x,t)$ is not in $A$, and \eqref{eq4.10} 
and \eqref{eq4.12} imply that $(x,t)$ is not in $B$, then $(x,t)$ must lie in the tubular 
neighborhood $\mathcal{N}_{\eta}$. Therefore, for any $\epsilon\in(0,\epsilon_3)$,
\begin{equation}\label{eq4.13}
\sup_{x\in\Gamma_t^{\epsilon,h,k}} \mbox{\rm dist}(x,\Gamma_t) \leq\eta \qquad\mbox{uniformly on $[0,T]$}.
\end{equation}
The proof is complete.
\end{proof}

\section{Numerical experiments}\label{sec-5}
In this section, we present three two-dimensional numerical tests to gauge the performance 
of the proposed fully discrete MIP-DG methods using the linear element (i.e., $r=1$). The square 
domain $\Omega=[-1,1]^2$ is used in all three tests and the initial condition is chosen to have
the form $u_0=\tanh\bigl(\frac{d_0(x)}{\sqrt{2}\epsilon}\bigr)$, where $d_0(x)$ denotes the signed 
distance from $x$ to the initial interface $\Gamma_0$.

Our first test uses a smooth initial condition to satisfy the requirement for $u_0$, consequently,
the theoretical results established in this paper apply to this test problem.
On the other hand, non-smooth initial conditions are used in the second and third tests, 
hence, the theoretical results of this paper may not apply. But we still use our MIP-DG methods
to compute the error order, energy decay and the evolution of the numerical interfaces. 
Our numerical results suggest that the proposed DG schemes work well, even a convergence 
theory is missing for them.

\medskip
$\mathbf{Test\, 1.}$ Consider the Cahn-Hilliard problem (\ref{eq1.1})-(\ref{eq1.4}) 
with the following initial condition:
{\small
\begin{equation*}
 u_0(x)=\tanh\Bigl(\frac{d_0(x)}{\sqrt{2}\eps}\Bigr),
\end{equation*}
}
where $\tanh(t)=(e^t-e^{-t})/(e^t+e^{-t})$, and $d_0(x)$ represents the signed distance function 
to the ellipse:
\begin{equation*}
\frac{x_1^2}{0.36}+\frac{x_2^2}{0.04}=1.
\end{equation*}
Hence, $u_0$ has the desired form as stated in Proposition \ref{prop2.3}.

Table \ref{tab3} shows the spatial $L^2$ and $H^1$-norm errors and convergence rates, 
which are consistent with what are proved for the linear element in the convergence theorem. 
$\epsilon=0.1$ is used to generate the table.
\begin{table}[htb]
\begin{center}
\begin{tabular}{|l|c|c|c|c|}
\hline
& $L^\infty(L^2)$ error & $L^\infty(L^2)$ order& $L^2(H^1)$ error& $L^2(H^1)$ order\\ \hline
$h=0.4\sqrt{2}$ & 0.53325 & & 0.84260 & \\ \hline
$h=0.2\sqrt{2}$ &0.21280  &1.3253 & 0.64843 &0.3779 \\ \hline
$h=0.1\sqrt{2}$ & 0.07164 &1.5707& 0.43273 &0.5835 \\ \hline
$h=0.05\sqrt{2}$ & 0.01779 & 2.0097& 0.21411 &1.0151 \\ \hline
$h=0.025\sqrt{2}$ & 0.00454 & 1.9703& 0.10890 & 0.9753\\ \hline
\end{tabular}
\smallskip
\caption{Spatial errors and convergence rates of Test 1 with $\epsilon=0.1$.} 
\label{tab3} 
\end{center}
\end{table}

Figure \ref{3energy} plots the change of the discrete energy $E_h(U^\ell)$ in time, which should decrease according to \eqref{eq3.15}.
This graph clearly confirms this decay property. 
\begin{figure}[tbh]
   \centering
   \includegraphics[width=3.5in,,height=2.0in]{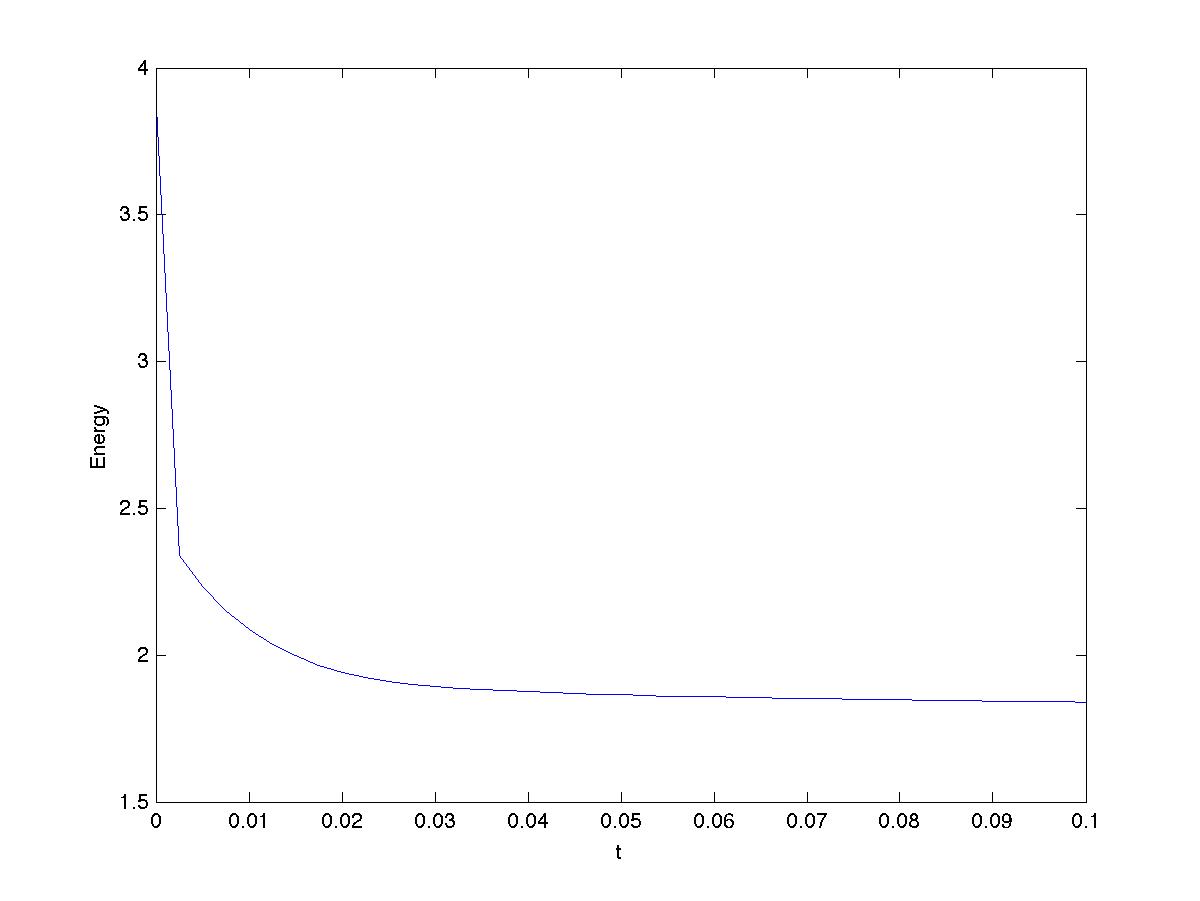} 
   \caption{Decay of the numerical energy $E_h(U^\ell)$ of Test 1.} \label{3energy}
\end{figure}
Figure \ref{figure9012} displays four snapshots at four fixed time points of the numerical interface 
with four different $\eps$.  They clearly indicate that at each time point the numerical interface 
converges to the sharp interface $\Gamma_t$ of the Hele-Shaw flow 
as $\epsilon$ tends to zero. It also shows that the numerical interface evolves faster in time for 
larger $\epsilon$ and confirms the mass conservation property of the 
Cahn-Hilliard problem as the total mass does not change in time, which approximates a constant 3.064.

\begin{figure}[th]
\centering
\includegraphics[height=1.8in,width=2.4in]{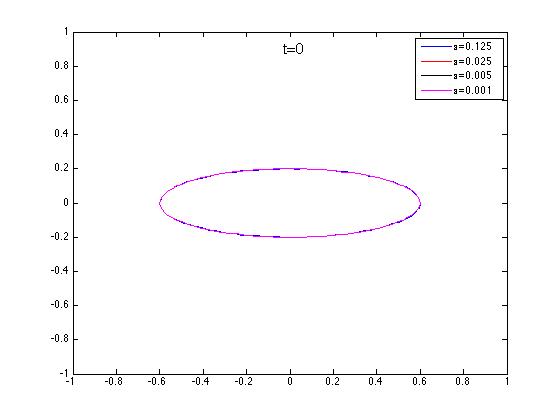}
\includegraphics[height=1.8in,width=2.4in]{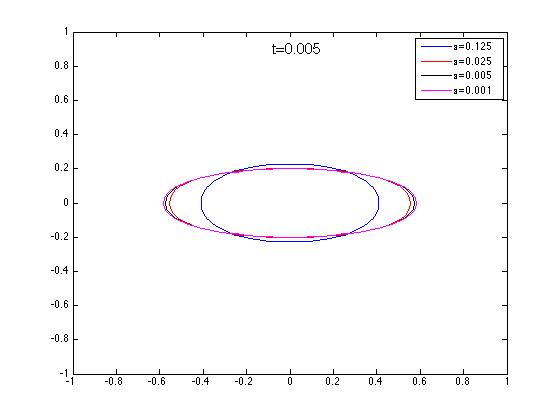}

\includegraphics[height=1.8in,width=2.4in]{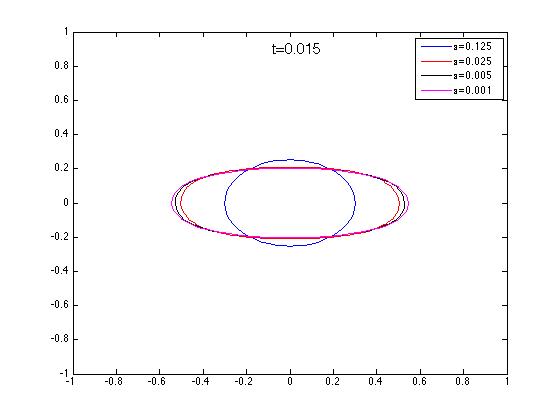}
\includegraphics[height=1.8in,width=2.4in]{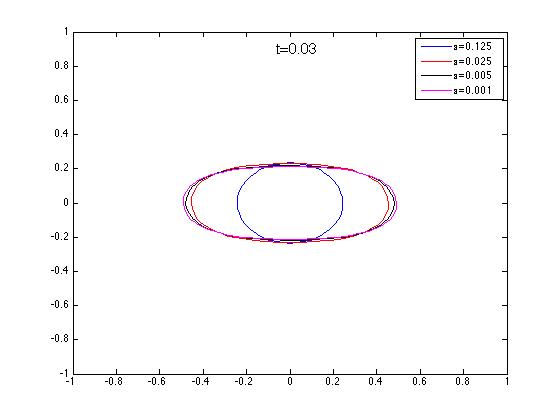}
\caption{Test 1: Snapshots of the zero-level set of $u^{\epsilon,h,k}$ at time $t=0, 
0.005, 0.015, 0.03$ and $\epsilon=0.125, 0.025, 0.005, 0.001$.}\label{figure9012}
\end{figure}

\medskip
$\mathbf{Test\, 2.}$ Consider the Cahn-Hilliard problem (\ref{eq1.1})-(\ref{eq1.4}) 
with the following initial condition:
{\small
\begin{equation*}
u_0(x)=\tanh\Bigl(\frac{1}{\sqrt{2}\eps}
\bigl(\min\bigl\{\sqrt{(x_1+0.3)^2+x_2^2}-0.3,\sqrt{(x_1-0.3)^2+x_2^2}-0.25\bigr\}\bigr)\Bigr).
\end{equation*}
}
We note that $u_0$ can be written as 
\[
u_0(x)=\tanh\Bigl(\frac{d_0(x)}{\sqrt{2}\eps} \Bigr).
\]
Here $d_0(x)$ represents the signed distance function. We note that $u_0$ does not
have  the desired form as stated in Proposition \ref{prop2.3}.

Table \ref{tab1} shows the spatial $L^2$ and $H^1$-norm errors and convergence rates, 
which are consistent with what are proved for the linear element in the convergence theorem. 
$\eps=0.1$ is used to generate the table.
\begin{table}[htb]
\begin{center}
\begin{tabular}{|l|c|c|c|c|}
\hline
& $L^\infty(L^2)$ error & $L^\infty(L^2)$ order& $L^2(H^1)$ error& $L^2(H^1)$ order\\ \hline
$h=0.4\sqrt{2}$ & 0.26713 & & 0.35714 & \\ \hline
$h=0.2\sqrt{2}$ &0.07161  &1.8993 & 0.18411 &0.9559 \\ \hline
$h=0.1\sqrt{2}$ & 0.01833 &1.9660& 0.09620 &0.9365 \\ \hline
$h=0.05\sqrt{2}$ & 0.00476 & 1.9452& 0.04928 &0.9650 \\ \hline
$h=0.025\sqrt{2}$ & 0.00121 & 1.9760& 0.02497 & 0.9808\\ \hline
\end{tabular}
\smallskip
\caption{Spatial errors and convergence rates of Test 2 with $\epsilon=0.1$.} 
\label{tab1} 
\end{center}
\end{table}
Figure \ref{test1_energy} plots the change of the discrete energy
$E_h(U^\ell)$ in time, which should decrease according to \eqref{eq3.15}.
This graph clearly confirms this decay property. 
\begin{figure}[tbh]
   \centering
   \includegraphics[width=3.5in,,height=2.0in]{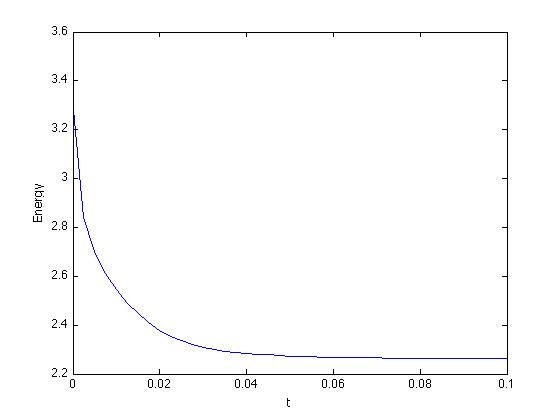} 
   \caption{Decay of the numerical energy $E_h(U^\ell)$ of Test 2.} \label{test1_energy}
\end{figure}
Figure \ref{figure1234} displays four snapshots at four fixed time points of the numerical 
interface with four different $\eps$. They clearly indicate that at each time point the 
numerical interface converges to the sharp interface $\Gamma_t$ of the Hele-Shaw flow 
as $\epsilon$ tends to zero. It again shows that the numerical interface evolves faster in time for 
larger $\epsilon$ and confirms the mass conservation property of the
Cahn-Hilliard problem as the total mass does not change in time, which approximates a constant 3.032.

\begin{figure}[th]
\centering
\includegraphics[height=1.8in,width=2.4in]{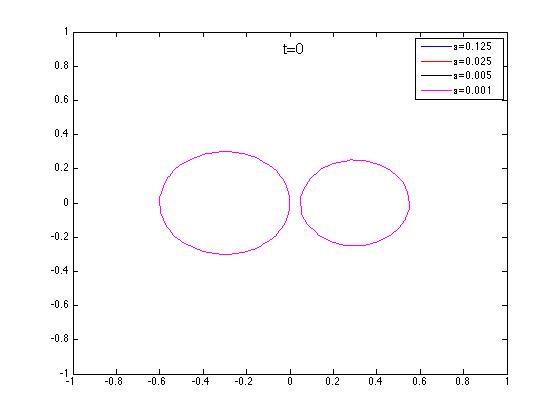}
\includegraphics[height=1.8in,width=2.4in]{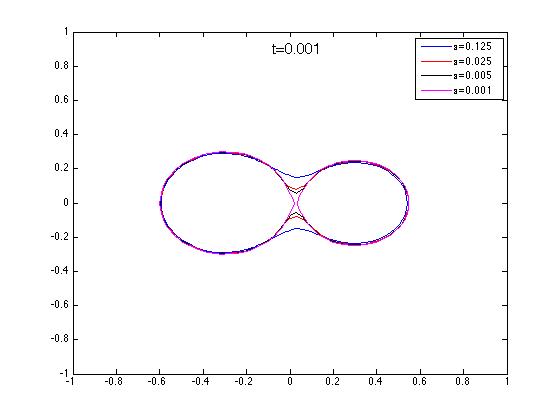}

\includegraphics[height=1.8in,width=2.4in]{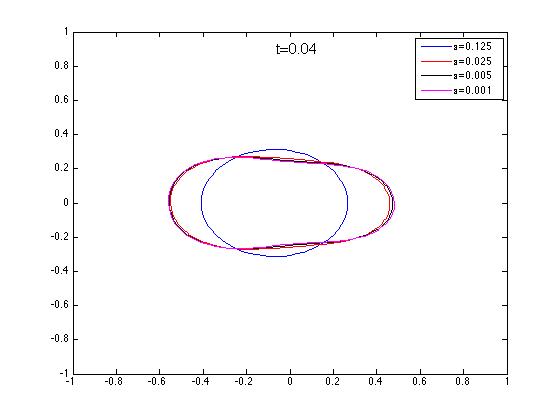}
\includegraphics[height=1.8in,width=2.4in]{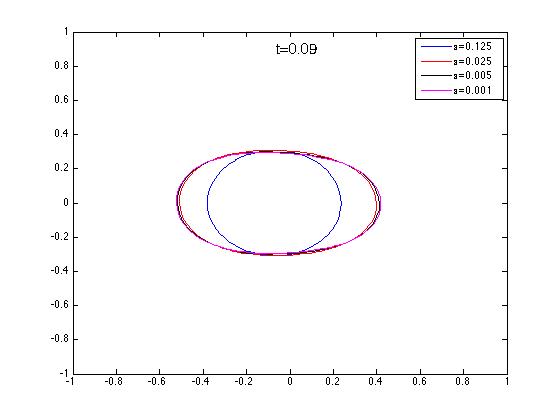}
\caption{Test 2: Snapshots of the zero-level set of $u^{\epsilon,h,k}$ at time $t=0, 
0.001, 0.04, 0.09$ and $\epsilon=0.125, 0.025, 0.005, 0.001$.}\label{figure1234}
\end{figure}

\medskip
$\mathbf{Test\, 3.}$ Consider the Cahn-Hilliard problem \eqref{eq1.1}--\eqref{eq1.4}
with the following initial condition:
\begin{align*}
u_0(x)=\tanh\Bigl(\frac{1}{\sqrt{2}\eps}\bigl(\min\bigl\{\sqrt{(x_1+0.3)^2+x_2^2}-0.2,\sqrt{(x_1-0.3)^2+x_2^2}-0.2,\\
\sqrt{x_1^2+(x_2+0.3)^2}-0.2,\sqrt{x_1^2+(x_2-0.3)^2}-0.2\bigr\}\bigr)\Bigr).
\end{align*}
Notice that the above $u_0$ does not have the desired form as stated in Proposition \ref{prop2.3}.
 
Table \ref{tab2} shows the spatial $L^2$ and $H^1$-norm errors and convergence rates with $\eps=0.1$, 
which are consistent with what are proved for the linear element in the convergence theorem. 
\begin{table}[t]
\begin{center}
\begin{tabular}{|l|c|c|c|c|}
\hline
& $L^\infty(L^2)$ error & $L^\infty(L^2)$ order& $L^2(H^1)$ error& $L^2(H^1)$ order\\ \hline
$h=0.4\sqrt{2}$ & 0.38576 & &0.84157  & \\ \hline
$h=0.2\sqrt{2}$ &0.12347  &1.6435 &0.55082  & 0.6115\\ \hline
$h=0.1\sqrt{2}$ & 0.03599 &1.7785& 0.31149 & 0.8224\\ \hline
$h=0.05\sqrt{2}$ & 0.00965 & 1.8990& 0.16199 &0.9433 \\ \hline
$h=0.025\sqrt{2}$ & 0.00247 &1.9660 & 0.08218 & 0.9790\\ \hline
\end{tabular}
\smallskip
\caption{Spatial errors and convergence rates of Test 3 with $\epsilon=0.1$.} 
\label{tab2} 
\end{center}
\end{table}
Figure \ref{test2_energy} plots the change of the discrete energy
$E_h(U^\ell)$ in time, which again decreases as predicted by \eqref{eq3.15}.
\begin{figure}[tbh]
   \centering
   \includegraphics[width=3.5in,,height=2.0in]{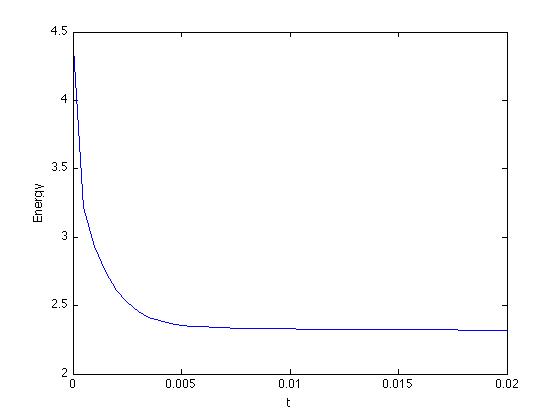} 
   \caption{Decay of the numerical energy $E_h(U^\ell)$ of Test 3.} \label{test2_energy}
\end{figure}
Figure \ref{figure5678} displays four snapshots at four fixed time 
points of the numerical interface with four different $\epsilon$.  Once again, we observe that at each time 
point the numerical interface converges to the sharp interface $\Gamma_t$ of the Hele-Shaw flow  
as $\epsilon$ tends to zero, the interface evolves faster in time for larger $\epsilon$ 
and the mass conservation property is preserved. The total mass approximates a constant 2.989.

\begin{figure}[th]
\centering
\includegraphics[height=1.8in,width=2.4in]{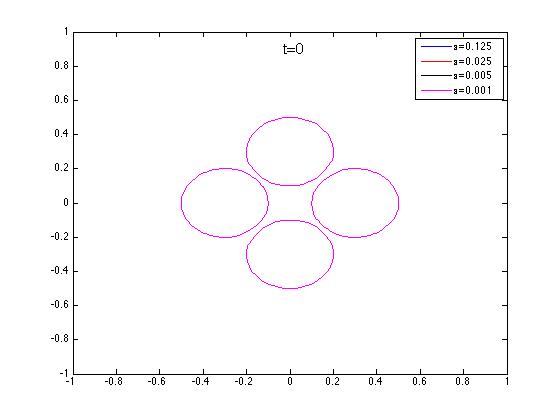}
\includegraphics[height=1.8in,width=2.4in]{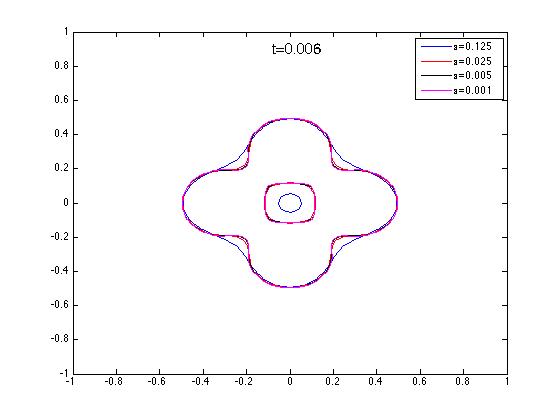}

\includegraphics[height=1.8in,width=2.4in]{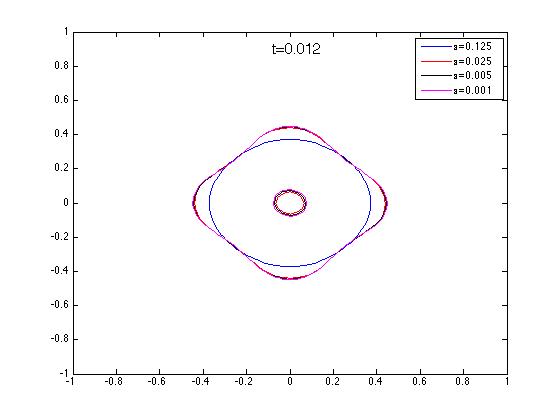}
\includegraphics[height=1.8in,width=2.4in]{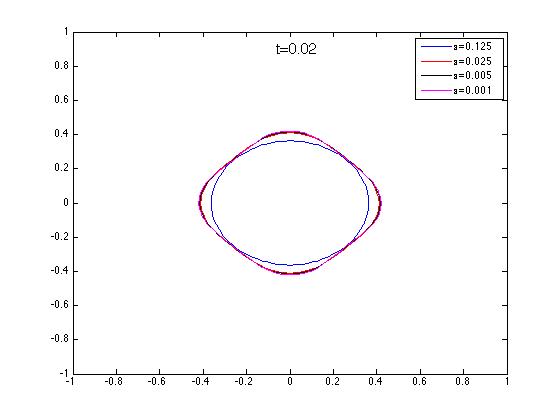}
\caption{Test 3: Snapshots of the zero-level set of $u^{\epsilon,h,k}$ at time $t=0, 
0.006, 0.012, 0.02$ and $\epsilon=0.125, 0.025, 0.005, 0.001$.}
\label{figure5678}
\end{figure}


\end{document}